\documentclass[12pt,a4paper]{article}
\usepackage[latin1]{inputenc}
\usepackage{amsmath}
\usepackage{amsfonts}
\usepackage{amssymb}
\usepackage{graphicx}
\usepackage{amsthm}
\usepackage[width=17.00cm, height=25.00cm]{geometry}
\usepackage{color}
\usepackage{bm}

\newtheorem{theorem}{Theorem}[section]
\newtheorem{lemma}[theorem]{Lemma}

\newtheorem{remark}[theorem]{Remark}

\newcommand{\norm}[2]{\|{#1}\|_{{#2}}}
\newcommand{\tnorm}[1]{\|\kern-.4mm| {#1} |\kern-.4mm\|}
\newcommand{\ud}{\,{\rm d}}
\newcommand{\jump}[1]{[\kern-0.6mm [{#1}]\kern-0.6mm]}
\newcommand{\av}[1]{\{\kern-1.5mm \{{#1}\}\kern-1.5mm\}}

\newcommand{\mbf}[1]{\bm{{#1}}}
\newcommand{\man}[1]{{\color{black}{#1}}} 

\newcommand{\change}[1]{{#1}}

\author{Emmanuil H.~Georgoulis \thanks{School of Mathematics and Actuarial Science, University of Leicester, Leicester LE1 7RH, United Kingdom. Email: {\tt{Emmanuil.Georgoulis@le.ac.uk}}} \thanks{Department of Mathematics, School of Applied Mathematical and Physical Sciences, National Technical University of Athens, Zografou 15780, Greece
} \thanks{Institute of Applied and Computational Mathematics, Foundation of Research and Technology -- Hellas, Vassilika Vouton 700 13 Heraklion,  Crete, Greece
} }
\title{Hypocoercivity-compatible finite element methods \\ for the long-time computation of Kolmogorov's equation}
\begin{document}
	\maketitle
	
	\begin{abstract}
		This work is concerned with the development of a family of Galerkin finite element methods for the classical Kolmogorov equation. Kolmogorov's equation serves as a sufficiently rich, for our purposes, model problem for kinetic-type equations and is characterised by diffusion in \emph{one} of the two (or three) spatial directions only. Nonetheless, its solution admits typically decay properties to some long time equilibrium, depending  on closure by suitable boundary/decay-at-infinity conditions. A key attribute of the proposed family of methods is that they also admit similar decay properties at the (semi)discrete level for very general families of triangulations. The method construction uses ideas by the general theory of \emph{hypocoercivity} developed by Villani \cite{villani}, along with judicious choice of numerical flux functions.  These developments turn out to be sufficient to imply that the proposed finite element methods admit a priori error bounds with constants \emph{independent of the final time}, despite Kolmogorov equation's degenerate diffusion nature. Thus, the new methods provably allow for robust error analysis for final times tending to infinity. The extension to three spatial dimensions is also briefly discussed. 
		\end{abstract}
	
	\section{Introduction}
 Degenerate parabolic problems in kinetic modelling are often characterised by the explicit presence of diffusion/dissipation in \emph{some} of the spatial directions only, yet may still admit decay properties to some long time equilibrium. The development of Galerkin finite element methods for such problems satisfying provably also such decay properties is, remarkably, an unexplored area. This is despite the potential attraction  that general, possibly highly non-uniform or adaptive, grid approximations can bring to kinetic simulations, especially those concerning expansive long-time simulation phenomena. Aiming to contribute in this direction, we shall develop and analyse a new class of Galerkin finite element methods for a simple kinetic equation, equipped with suitable closures by known boundary/periodicity conditions from the literature. To this end, we shall be concerned with the classical Kolmogorov partial differential equation (PDE)
	\begin{equation}\label{eq:kolmogorov}
	\mathcal{L}u\equiv u_t-u_{xx}+xu_y = f,\quad \text{in }(0,\infty)\times \man{\Omega},
	\end{equation}
\man{with $\Omega\subset\mathbb{R}^2$,} for suitably smooth forcing $f\in L_2(\mathbb{R}_+;H^1(\Omega))$, and Cauchy initial data $u(0,\cdot)=u_0\in L_2(\Omega)$. The notational conventions used above are somewhat non-standard in kinetic theory, whereby $\mathcal{L}$ may be written as $\mathcal{L}f:=f_t-f_{vv}+vf_x$ with $v$ denoting the particle velocity variable, $x$ the displacement/position and $f$ the respective probability density function. 
Kolmogorov showed already in 1934 that \eqref{eq:kolmogorov} admits a smooth fundamental solution outside the pole and, hence, it is hypoelliptic, i.e., it admits a smooth solution in $(0,\infty)\times \mathbb{R}^2$ for any smooth initial and forcing data $u_0,f$, respectively \cite{kolmogorov}. This may be a surprising assertion at first sight, since there is no explicit dissipation built into the PDE in the $y$-variable. The advection field $(0,x)^T$, however, is ``appropriately non-constant'' so that it suffices to propagate the built-in dissipation in the $x$-variable onto the entire spatial domain.  In 1967, H\"ormander, in the celebrated work \cite{hormander}, gave a sufficient condition for hypoellipticity for a wide class of 2nd order operators with non-negative characteristic form, which are nowadays often referred to as \emph{H\"ormander sum-of-squares operators}; interestingly, a motivating example for the developments in \cite{hormander} has, indeed, been Kolmogorov's equation \eqref{eq:kolmogorov}.  Hence, although \eqref{eq:kolmogorov} may appear to be a rather special equation at first sight, its significance as a model problem is paramount, for it encompasses a number of pertinent structural properties of large classes of kinetic models; we refer, e.g., to \cite{villani} for very instructive expositions.

The proof of trend to equilibrium for the directly related inhomogeneous Fokker-Planck equation has been given by H\'erau \& Nier in  \cite{herau_nier} upon realising that certain entropies admitting mixed derivatives ($u_{xy}$ or, respectively, $f_{vx}$) give rise to full gradients in certain weighted Sobolev spaces, from which Poincar\'e-Friedrichs inequalities (also known as spectral gap properties in this context) are sufficient to prove decay to an equilibrium distribution. Related ideas have also been used by Eckmann \& Hairer in \cite{eckmann_hairer} to study the spectral properties of certain hypoelliptic operators involving degenerate diffusions and by Mouhot \& Neumann \cite{mouhot_neumann} in the study of kinetic models with integral-type collision operators, among others. The idea of using entropies involving mixed derivatives was elevated to a general framework in proving decay to equilibrium for kinetic equations by Villani \cite{villani} via the introduction of the concept of \emph{hypocoercivity}. Roughly speaking, hypocoercivity is the property of certain degenerate elliptic or parabolic differential operators to yield dissipation of the solution also in the directions where \emph{no} diffusion is explicitly present. Astonishingly, sufficient conditions for hypocoercivity are given by the rank spanned by H\"ormander vector fields and their commutators with respect to the vector field related to 1st order term \cite{villani}. Thus, sufficient conditions for hypocoercivity are closely related, but \emph{not} identical, to the classical H\"ormander's rank condition for hypoellipticity.

To fix ideas,  denoting by $(\cdot,\cdot)_{L_2}$ the inner product of $L_2$, \change{over some $\Omega\subset \mathbb{R}^2$, and respective norm $\norm{\cdot}{L_2}:=\sqrt{(\cdot,\cdot)}_{L_2}$ and, assuming suitable boundary conditions}, the classical energy-type analysis for \eqref{eq:kolmogorov} yields
\begin{equation}\label{eq:energy_classical}
\frac{1}{2}\frac{\ud}{\ud t}\norm{u(t)}{\change{L_2}}^2+\norm{u_x}{\change{L_2}}^2= (f,u)_{\change{L_2}},
\end{equation}
for almost all $t\in(0,t_f]$, for some final time $t_f\in\mathbb{R}_+$, upon observing that the skew-symmetric term $(xu_y,u)$ vanishes. A combination of Cauchy-Schwarz' and Young's inequalities on the last term on the right-hand side \change{of} \eqref{eq:energy_classical} \man{thus} yields
\begin{equation}
\int_0^{t_f}(f,u)_{\change{L_2}}\ud t\le \frac{1}{2\delta}\norm{f}{L_2(0,t_f;\change{L_2})}^2+ \frac{\delta}{2}\norm{u}{L_2(0,t_f;\change{L_2})}^2 ,
\end{equation}
for any $\delta>0$. Therefore, Gr\"onwall's Lemma implies
\begin{equation}\label{eq:basic_energy_estimate}
\norm{u(t_f)}{\change{L_2}}^2\le e^{\delta t_f} \big(\delta^{-1}\norm{f}{L_2(0,t_{\change{f}};\change{L_2})}^2+\norm{u_0}{\change{L_2}}^2\big).
\end{equation}
Since the right-hand side of \eqref{eq:basic_energy_estimate} is always greater than or equal to $
  \norm{u_0}{\change{L_2}}^2
$
for every $f\in L_2(0,\infty;\change{L_2})$, we conclude that the basic energy estimate may be severely pessimistic as $t_f\to \infty$ for any fixed $\delta>0$: for any finite final time $t_f>0$, the above implies that the right-hand side of \eqref{eq:basic_energy_estimate} will be positive away from zero for any non-trivial initial condition. Notice that when $f=0$, we can take $\delta\to 0$, yielding the stability estimate $\norm{u(t_f)}{\change{L_2}}\le \norm{u_0}{\change{L_2}}$. This is at odds with the fast decay to $0$ observed for this problem by the solution of Kolmogorov's equation \change{when equipped with suitable conditions, e.g., periodic boundary conditions or decay-to-infinity conditions when $\Omega=\mathbb{R}^2$}. Evidently the key culprit in the above analysis is the absence of control in a time-integrated norm on the left-hand side of \eqref{eq:basic_energy_estimate}. The situation does not appear to improve if different versions of Gr\"onwall-type inequalities are applied instead. 

On the other hand, for the non-degenerate parabolic equation
\begin{equation}
	\tilde{\mathcal{L}}u\equiv u_t-u_{xx}-u_{yy}+xu_y = f,\quad \text{in }(0,\infty)\times \man{\Omega},
\end{equation}
with the same forcing and initial conditions, we can easily deduce the energy identity
\begin{equation}\label{eq:energy_classical_tilde}
\frac{1}{2}\frac{\ud}{\ud t}\norm{u(t)}{\change{L_2}}^2+\norm{\nabla u}{\change{L_2}}^2= (f,u)_{\change{L_2}}.
\end{equation}
Assuming now the validity of a Poincar\'e-Friedrichs inequality of the form $\norm{u}{\change{L_2}}^2\le C_{PF}\norm{\nabla u}{\change{L_2}}^2$, and working as before, we can arrive at the estimate
\begin{equation}\label{eq:energy_classical_tilde_two}
\frac{1}{2}\frac{\ud}{\ud t}\norm{u(t)}{\change{L_2}}^2+C_{PF}^{-1}\norm{ u}{\change{L_2}}^2\le  (f,u)_{\change{L_2}},
\end{equation}
from which, Gr\"onwall's Lemma implies
\begin{equation}\label{eq:basic_energy_estimate_nondeg}
\norm{u(t_f)}{\change{L_2}}^2 \le {\rm e}^{-C_{PF}^{-1}t_f} \norm{u_0}{\change{L_2}}^2+C_{PF}\int_0^ {t_f}{\rm e}^{-C_{PF}^{-1}(t_f-t)} \norm{ f(t)}{\change{L_2}}^2\ud t,
\end{equation}
whose right-hand side clearly decays exponentially as $t_f\to\infty$ when $f=0$, for instance. This is in sharp contrast with \eqref{eq:basic_energy_estimate}.

\change{A key idea in the context of hypocoercivity, on the other hand, is to arrive at a modified energy setting whereby Poincar\'e-Friedrichs inequality is valid for the \emph{original} problem \eqref{eq:kolmogorov}. In our specific setting of Kolmogorov's equation, upon considering carefully constructed families of Hilbertian norms $\sqrt{(\cdot,\cdot)_{\rm hc}}=:\norm{\cdot}{\rm hc}:=\big(\norm{\cdot}{L_2}^2+\norm{\sqrt{A}\nabla\cdot}{L_2}^2\big)^{1/2}$, where $A\in \mathbb{R}^{2\times2}_{\rm symm}$ are suitable positive semidefinite  matrices, it is possible to arrive at an energy estimate of the form
	\begin{equation}\label{eq:energy_mod}
	\frac{1}{2}\frac{\ud}{\ud t}\norm{u(t)}{\rm hc}^2+c\big(\norm{\nabla u}{L_2}^2+\norm{u_x}{\rm hc}^2)\le (f,u)_{\rm hc},
	\end{equation}
	for suitable/compatible boundary conditions, with $c$ a positive constant. Crucially, if $u$ is such that it satisfies a Poincar\'e-Friedrichs inequality, we can arrive to an estimate of the form \eqref{eq:basic_energy_estimate_nondeg} for the original, degenerate Kolmogorov's equations. As a result, in the homogeneous case $f=0$ \eqref{eq:basic_energy_estimate_nondeg} implies exponential decay as $t_f\to\infty$.}

\change{Obviously, \eqref{eq:basic_energy_estimate}} is not a satisfactory state of affairs not only at the PDE level, but also under the prism of the numerical analysis for these problems. Indeed, almost all provably convergent, general numerical methods which can be potentially applied to this class of PDEs use standard energy-type arguments \change{as in \eqref{eq:basic_energy_estimate}} for their error analysis, resulting to possibly severely pessimistic a priori and/or a posteriori error bounds for large final times $t_f$. \man{Available} Galerkin-type numerical methods for (degenerate) second order evolution problems of diffusion type \man{typically also suffer in this respect}; we non-exhaustively refer to \cite{act,hss,eg,cdgh_book} as representative works presenting different approaches. \change{A numerical approach accommodating instead the functional setting of \eqref{eq:energy_mod}, thereby taking advantage of hypocoercivity properties is, therefore, desirable.}

To the best of our knowledge, there exist only few, yet quite inspiring, exceptions which employ hypocoercivity ideas in the context of numerical methods. In particular, the work of Porretta \& Zuazua \cite{porretta_zuazua} proves the hypocoercivity of classical finite difference operators discretising Kolmogorov's equation for $f=0$ on uniform spatial grids over $[0,\infty)\times\mathbb{R}^2$. The recent manuscript by Dujardin, H\'erau, \& Lafitte \cite{dujardin_herau_lafitte} takes the approach of \cite{porretta_zuazua} further by establishing discrete versions of weighted Poincar\'e inequalities for difference operators, thereby showing decay to equilibrium for finite difference approximations of the inhomogeneous two-dimensional Fokker-Planck equation, which is, of course, closely related to Kolmogorov's equation. Both these approaches apply directly the abstract semigroup result of Villani \cite[Theorem 18]{villani} to the finite difference operators used by proving the necessary commutator properties required for its validity. With regard to practical \man{works}, Foster, Loh\'eac \& Tran \cite{foster_loheac_tran} discuss the development of a Lagrangian-type splitting method based on a carefully constructed similarity transformation and linear finite elements over quasiuniform meshes, although no proof of decay to equilibrium is given for the numerical method itself. Also, Bessemoulin-Chatard \& Filbet \cite{besse_filbet} present design principles for the construction of equilibrium-preserving finite volume methods for nonlinear degenerate parabolic problems. \change{Recently, Bessemoulin-Chatard, Herda \& Rey \cite{besse_FV} presented an asymptotic-preserving in the diffusive limit  finite volume scheme for a class of one-dimensional kinetic equations, using the hypocoercivity framework of Dolbeault, Mouhot \& Schmeiser \cite{DMS_TAMS}.} 

For practical computations, one is typically forced to confine the spatial computational domain into an open and bounded one, say $\Omega\subset \mathbb{R}^d$, $d=2,3$. If the objective is to compute equilibrium states for the whole of $\mathbb{R}^d$,  the truncation of the computational domain and the imposition of boundary conditions should be performed carefully, for it may alter the rate of decay to equilibrium compared to the non-confined problem \cite{foster_loheac_tran}. If one is interested instead in computations on bounded domains as such, then the modelling often dictates the imposition of additional boundary conditions/constraints, such as \change{solution} periodicity or far field decay. \change{Indeed, as hypocoercivity is manifested via certain properties of commutators of directional derivatives appearing in the differential operator, integration by parts is an essential tool. As a result when confined on bounded domains the very nature of boundary conditions enforced becomes essential.} As we shall see below, the imposition of boundary conditions in a fashion allowing concepts of hypocoercivity to be ported to bounded spatial domains is a challenge both in the PDE and the numerical analysis contexts. \change{Indeed, to the best of our knowledge \emph{no} exhaustive study of hypocoercivity-inducing boundary conditions for problems posed on bounded spatial domains has been carried out in the respective PDE literature.}

This work is concerned with the development of a Galerkin-type finite element method which is compatible with hypocoercivity concepts and, therefore, allows for the proof of error bounds whose error constants \emph{do not} grow as the final time $t_f\to\infty$. \man{This is in sharp contrast to estimates derived via standard energy arguments as discussed above; \change{cf}.~\eqref{eq:basic_energy_estimate}. In particular,} as we shall show below, there will be \emph{no dependence} of the error constant on the final time $t_f$ \change{in the setting of \eqref{eq:energy_mod}}. This, in turn, will \change{highlight theoretically any} potential advantages of the numerical methods presented below for long time computations.  The construction of the finite element method is based on a variational interpretation of the modified entropies involving mixed derivatives of H\'erau \& Nier \cite{herau_nier} and Villani \cite{villani}, which are able to reveal the implicit diffusion properties in the $y$-variable also. The derived modified variational problem is based on the PDE and higher order laws stemming from the PDE and involving up to 4th order partial derivatives, when viewed formally in strong form. The finite element method is constructed via a Galerkin projection onto a suitable $C^0$-finite element space subordinate to a\man{, possibly, highly} non-uniform triangulation. The proposed method is, nevertheless, non-conforming due to the 4th order symbol of the variational problem. To recover, therefore, the consistency of the finite element method without compromising on the hypocoercivity properties, the variational form has to be enriched by carefully constructed consistent numerical fluxes over the skeleton of the triangulation. The proposed finite element method is proven to decay to equilibrium when $f=0$, when we equip the PDE with additional natural boundary conditions; this is, to the best of our knowledge, the first manifestation of hypocoercivity for a Galerkin scheme in the literature. Moreover, we prove a priori error bounds in the natural norm determined by the hypocoercivity result.

We view the developments presented below as proof of concept that Galerkin finite element methods over general non-uniform grids can be constructed in a compatible fashion to retain hypocoercivity properties at discrete level. This, in turn, leads to a priori error bounds which are \emph{robust} with respect to their dependence on the final time $t_f$. Kolmogorov's equation has the role of a sufficiently rich model problem to highlight a new finite element methodology for hypocoercive PDEs on bounded computational domains with view to retaining the long-time behaviour of the exact initial/boundary value problems. Apart from the mathematical challenge, the development of numerical methods for kinetic equations on general, possibly non-quasi-uniform meshes has important practical ramifications. Indeed, modern applications of kinetic equations involve the movement of large particles for which Lagrangian schemes appear to be preferential by practitioners which greatly benefits from the ability to discretise over unstructured, possibly moving, meshes.

The remainder of this work is structured as follows. In Section \ref{sec:weak_form}, we introduce a variational formulation motivated by design concepts of hypocoercivity \cite{villani}, which is shown to be coercive with respect to an $H^1$-equivalent norm in Section \ref{sec:hypo}. A discrete bilinear form using continuous, yet non-conforming, elements with appropriately constructed numerical fluxes ensuring the coercivity in the discrete setting also is proposed in Section \ref{sec:fem} and is shown to admit completely analogous decay to equilibrium properties as the respective continuous problem in the non-external forcing scenario. The error analysis of the proposed finite element method is given in Section \ref{sec:error_analysis} along with a discussion on its properties. \change{Finally, we present a series of numerical experiments highlighting the predicted convergence rates by the theory; this is the content of Section \ref{sec:num}.}

To simplify notation, we shall abbreviate the $L_2(\omega)$-inner product and $L_2(\omega)$-norm for a Lebesgue-measurable subset $\omega\subset \mathbb{R}^d$ as $(\cdot,\cdot)_{\omega}$ and $\norm{\cdot}{\omega}$, respectively. Moreover, when $\omega=\Omega$, we shall further compress the notation to $(\cdot,\cdot)\equiv (\cdot,\cdot)_{\Omega}$ and $\norm{\cdot}{}\equiv \norm{\cdot}{\Omega}$. We shall also make use of the standard notation $H^k(\omega)$ for Hilbertian Sobolev spaces, $k\in\mathbb{R}$.

\section{A special weak formulation}\label{sec:weak_form}

The problem of closing a degenerate parabolic PDE with suitable boundary conditions is well understood via the classical theory of linear second order equations with non-negative characteristic form \cite{fichera,or73}. In particular, with $\mathbf{n}(\cdot):=(n_1(\cdot),n_2(\cdot))^T$ denoting the unit outward normal vector at almost every point of $\partial\Omega$, which is assumed to be piecewise smooth and Lipschitz, we first define the elliptic portion of the boundary
\[
\partial_0\Omega:=\{(x,y)\in \partial \Omega: n_1(x,y)\ne 0\} {\color{blue}.}
\]
On the non-elliptic   portion of the boundary $\partial \Omega\backslash\partial \Omega_0$, we define the inflow and outflow boundaries: 
\[
\partial_-\Omega:= \{(x,y)\in \partial\Omega\backslash\partial_0\Omega: xn_2(x,y)<0\}, \quad \partial_+\Omega:= \{(x,y)\in  \partial\Omega\backslash\partial_0\Omega: xn_2(x,y)   \ge  0\}.
\]
(For notational brevity, we have included any characteristic portions  $\{(x,y)\in  \partial\Omega\backslash\partial_0\Omega: xn_2(x,y)  = 0\}$ of the boundary  into the outflow part, since their treatment is identical in what follows.)
Introducing the notation $ Lu\equiv -u_{xx}+xu_y$, we consider the initial/boundary-value problem:
	\begin{equation}\label{eq:kolmogorov-bounded}
	\begin{aligned}
u_t+Lu= u_t -u_{xx}+xu_y =&\ f, &\text{in } (0,t_f]\times\Omega,\\
u=&\ u_0, &\text{on } \{0\}\times\Omega,\\
u=&\ 0, &\text{on } (0,t_f]\times\partial_{-,0}\Omega,
\end{aligned}
\end{equation}
with $\partial_{-,0}\Omega:=\partial_{-}\Omega\cup\partial_{0}\Omega$, for $t_f>0$ and for $f\in H^1(\Omega)$,  noting the non-standard regularity assumption for $f$. The well-posedness of the above problem is assured upon assuming that $\partial_{-,0}\Omega$ has positive one-dimensional Hausdorff measure \cite{or73}. It is also possible to impose Neumann-type boundary conditions on parts of $\partial_0\Omega$; this is not done here in the interest of simplicity of the presentation only. For a  three-dimensional counterpart of the above model problem we refer to Remark \ref{remark_3d}.

Let $A:=\left(\begin{array}{cc} \alpha & \beta\\ \beta & \gamma \end{array}\right) 
$ be
a symmetric and non-negative definite matrix with $\alpha,\beta,\gamma$ non-negative parameters, and set
\begin{equation}\label{eq:bigA}
\delta:=1-\frac{\beta}{\sqrt{ \alpha\gamma}};
\end{equation} 
the fact that $A$ is non-negative definite implies that $0\le\delta\le 1$. Assuming sufficient regularity for the exact solution $u$, so that the following calculations are well defined, \eqref{eq:kolmogorov} implies
$
\nabla u_t + \nabla L u = \nabla f,
$
which, tested against $A\nabla v$ for any $v\in H^1(\Omega)$, results in
\[
(\nabla u_t,A\nabla v) + (\nabla L u, A\nabla v)= (\nabla f,A\nabla v);
\]
here and below, we denote $(\cdot,\cdot)\equiv (\cdot,\cdot)_{L_2(\Omega)}$ and $\|\cdot\|\equiv\|\cdot\|_{L_2(\Omega)}$ for brevity. Resorting to \eqref{eq:kolmogorov} once more, after an integration by parts, we deduce the variational form
\begin{equation}\label{eq:weak_hypo}
(u_t,v)+(\nabla u_t,A\nabla v) + (u_x,v_x)+(xu_y,v)+(\nabla L u, A\nabla v)= (f,v)+(\nabla f,A\nabla v),
\end{equation}
for all $v\in H^1_{{-,0}}(\Omega):=\{v\in H^1(\Omega):v|_{\partial_{-,0}\Omega}=0\}$. The latter will be the basis of constructing a finite element method with the sought-after properties. To highlight the idea behind this construction in the context of hypocoercivity, we shall now discuss the properties of the spatial part of the operator on the left-hand side of \eqref{eq:weak_hypo}, viz., 
\begin{equation}\label{eq:spatial_part}
a(u,v):=(u_x,v_x)+(xu_y,v)+(\nabla L u, A\nabla v).
\end{equation}

\section{The hypocoercivity of $a(\cdot,\cdot)$}\label{sec:hypo}
We begin by somewhat modifying the approach presented in Villani \cite{villani} arriving, nevertheless, to an effectively similar result. For accessibility, we shall first revert to \eqref{eq:kolmogorov}, rather than \eqref{eq:kolmogorov-bounded}, i.e., we have $\Omega=\mathbb{R}^2$, assuming that $|u(x,y)|\to 0$ as $|(x,y)|\to\infty$; the case $\Omega\subset\mathbb{R}^2$ will be discussed below. 

Setting $v=u$ into \eqref{eq:spatial_part} gives
\[
a(u,u)=\|u_x\|^2+(xu_y,u)+(-u_{xxx}+(xu_y)_x,\alpha u_x+\beta u_y) + (-u_{xxy}+xu_{yy},\beta u_x+\gamma u_y),
\]
which, upon integration by parts, already yields
\[
\begin{aligned}
a(u,u)=&\ \|u_x\|^2+(u_{xx},\alpha u_{xx}+\beta u_{yx}) +((xu_y)_x,\alpha u_{x}+\beta u_{y}) \\
&+ (u_{xy},\beta u_{xx}+\gamma u_{yx})-(xu_{y},\beta u_{xy}+\gamma u_{yy})\\
=&\ \|u_x\|^2+\alpha\|u_{xx}\|^2+\beta\|u_y\|^2+\gamma\|u_{xy}\|^2+\alpha(u_y, u_x)+ 2\beta(u_{xy}, u_{xx}),
\end{aligned}
\]
noting that  $(xu_{xy}, u_{x})=0= (xu_{y}, u_{yy})$ from \change{the} divergence theorem and the decay properties of $u$ as $|(x,y)|\to \infty$. Young's and Cauchy-Schwarz' inequalities now imply
\[
a(u,u)\ge (1-\epsilon)\|u_x\|^2+\big(\alpha-\sqrt{\beta}\big)\|u_{xx}\|^2+\big(\beta-\frac{\alpha^2}{4\epsilon}\big)\|u_y\|^2+\big(\gamma-\beta^{\frac{3}{2}}\big)\|u_{xy}\|^2.
\]
Selecting $\alpha>0$, $\beta=4\alpha^2/9$, $\gamma=\alpha^3/3$, and $\epsilon=3/4$, for instance, we deduce
\begin{equation}\label{eq:hypoco}
a(u,u)\ge \frac{1}{4}\|u_x\|^2+\frac{\alpha}{3} \|u_{xx}\|^2+\frac{\alpha^2}{9}\|u_y\|^2+\frac{\alpha ^3}{27}\|u_{xy}\|^2.
\end{equation}
Another instructive choice is $\alpha>0$, $\beta=\alpha^2$, $\gamma=\alpha^3$ and $\epsilon=1/2$, giving
\begin{equation}\label{eq:hypoco2}
2a(u,u)\ge \|u_x\|^2+\alpha^2\|u_y\|^2.
\end{equation}
Thus, remarkably, using the non-standard test function $v-\nabla\cdot (A\nabla v)$ for well-chosen $A$, results into $L$ providing coercivity with respect to $u_y$ also! This is a manifestation of the concept of hypocoercivity. As briefly discussed above, the mechanics of this idea can be traced back to the celebrated work of H\"ormander \cite{hormander} for sufficient conditions for hypoellipticity for second order linear operators: the first order term is ``appropriately non-constant'' with respect to the independent variables $t,x,y$, allowing for control also in the commutator vector field $[\partial_x,x\partial_y]=\partial_y$, \change{using the usual notation $[A,B]:=AB-BA$ for a commutator of an algebraic structure}. Informally speaking, the first order vector field in conjunction with the direction of $\partial_x$ is able to ``propagate'' the $x$-variable dissipation of $L$ to the whole of $\mathbb{R}^2$, resulting to the property of hypocoercivity \cite{villani}. Notice that the choice of $\alpha,\beta,\gamma$ yielding \eqref{eq:hypoco} corresponds to $v-\nabla\cdot (A\nabla v)$ being an elliptic differential operator, while the respective choice for \eqref{eq:hypoco2} corresponds to the parabolic operator $v-\nabla\cdot (A\nabla v)$.  At this point, we prefer to remain general in the choice of $A$, very much following the point of view taken by Villani \cite{villani}. Some concrete choices of $A$ will be needed below to fully define the ensuing numerical method.

\change{The key second ingredient in proving solution behaviour as $t_f\to\infty$ is the availability of a Poincar\'e-Friedrichs/spectral gap inequality. Such inequalities typically hold for functions in standard Lebesgue spaces on bounded domains or for periodic ones, along with some boundary or moment conditions. For problems posed on whole Euclidean spaces, hypocoercivity is typically manifested in appropriately weighted Lebesque spaces \cite{villani}. In the present work, we are primarily interested in the mechanics of porting hypocoercivity into piecewise smooth Galerkin-type solutions; as such we shall assume that we have available compatible closure/boundary conditions for hypocoercivity to hold at the PDE level. }


\change{Boundary/closure conditions in} kinetic PDE models posed on bounded domains $\Omega\subset\mathbb{R}^d$ \change{are designed to}  convey \change{various} physical properties in the vicinity of boundaries, e.g., reflection, diffusion, periodicity, etc. In the case of the classical Kolmogorov equation considered here, the $x$-variable is typically associated with particle velocity, while the $y$-variable models position/displacement. Performing the above calculations on a bounded domain $\Omega$ instead, integrations by parts yield
\begin{equation}\label{eq:hypoco3}
\begin{aligned}
a(u,u)=&\ \norm{u_x}{}^2+\frac{1}{2}\norm{\sqrt{xn_2}u}{\partial_+\Omega}^2+\norm{\sqrt{A}\nabla u_x}{}^2+\beta\norm{u_y}{}^2+\alpha(u_y, u_x)\\
&+(\frac{1}{2}xn_2\nabla u-n_1\nabla u_{x}, A\nabla u)_{\partial\Omega}.
\end{aligned}
\end{equation}
To arrive again at \eqref{eq:hypoco}
(or \eqref{eq:hypoco2}), it is sufficient to \change{ensure} that 
\begin{equation}\label{eq:additional_bc}
(\frac{1}{2}xn_2\nabla u-n_1\nabla u_{x}, A\nabla u)_{\partial\Omega}\ge 0{\color{blue}.}
\end{equation}
Alternatively, it may be possible to control part of $(\frac{1}{2}xn_2\nabla u-n_1\nabla u_{x}, A\nabla u)_{\partial\Omega}$ by the non-negative terms on the right-hand side of \eqref{eq:hypoco3}, with the remaining part to satisfy a non-negativity condition.

\change{We now give some specific examples. For illustration,} we set $\Omega:=(x^-,x^+)\times(y^-,y^+)$ 
. Then,
 \[
\partial_0\Omega=\{(x^-,y),(x^+,y):\ y^-\le y< y^+\},  \quad  \partial_{\pm}\Omega=\{(x,y^{\pm}):\ x^-< x\le x^+\}.
 \]
\change{Further, we assume $u\in H^{5/2+\varepsilon}(\Omega)$-smoothness of the solution, $\varepsilon>0$ so that the last term on the right-hand side of \eqref{eq:hypoco3} is well-defined; this is a reasonable assumption in this context due to the hypoellipticity of Kolmogorov's equation \cite{hormander}.}

\change{As a first example, we remove the homogeneous boundary conditions on $\partial_{-,0}\Omega$ and, instead, we consider the family of sufficiently smooth functions on the torus $\mathbb{T}$. The continuity of partial derivatives implies that the last term on the right-hand side of \eqref{eq:hypoco3}, together with any additional boundary terms arising by the removal of the homogeneous boundary conditions, vanish.} 

\change{As a second example, we restrict the family of solutions to periodic functions in the diffusive $x$-direction, and we further assume that $u$ has vanishing inflow flux, viz., $u_y= 0$. 
	This gives}
 \begin{equation}\label{eq:periodic_bc}
(\frac{1}{2}xn_2\nabla u-n_1\nabla u_{x}, A\nabla u)_{\partial_0\Omega}=0,
 \end{equation}
since $xn_2=0$ on $\partial_0\Omega$. \change{Moreover, since $n_1=0$ and $u_y=0$ on $\partial_{\pm}\Omega$ and, observing that $u_x=0$ on $\partial_-\Omega$ (as $u_x$ is the tangential derivative on $\partial_-\Omega$ of the \emph{constant} initial condition),} we deduce
\[
(\frac{1}{2}xn_2\nabla u-n_1\nabla u_{x}, A\nabla u)_{\change{\partial_-}\Omega}= \frac{1}{2}(xn_2\nabla u, A\nabla u)_{\change{\partial_-}\Omega}
= 0 .
\]
%
%
\change{Therefore, we deduce
\[
(\frac{1}{2}xn_2\nabla u-n_1\nabla u_{x}, A\nabla u)_{\partial\Omega} =(\frac{xn_2}{2}\nabla u, A\nabla u)_{\partial_\pm\Omega} =\frac{1}{2}\norm{\sqrt{|xn_2| A}\nabla u}{\partial_+\Omega}^2.
\]}
We refer, e.g., to \cite{risken,dujardin_herau_lafitte} for further discussion on additional boundary conditions. Different combinations of assumptions  are also possible to yield \eqref{eq:additional_bc}, such as vanishing second derivative conditions across the elliptic boundary, etc. 

\change{
	\begin{remark}\label{remark_3d}
		To highlight the potential generality of the approach, we briefly consider the three-dimensional version of Kolmogorov's equation: for $\Omega\subset\mathbb{R}^3$, find $u:(0,t_f]\times \Omega\to\mathbb{R}$, such that
		\begin{equation}\label{eq:kolmogorov-bounded_dD}
		\begin{aligned}
		u_t+L_3u=u_t - u_{xx}+  xu_y+yu_z  =&\ f, &\text{in } (0,t_f]\times\Omega,\\
		u=&\ u_0, &\text{on } \{0\}\times\Omega,\\
		u=&\ 0, &\text{on } (0,t_f]\times\partial_{-,0}\Omega,
		\end{aligned}
		\end{equation}
		for $t_f>0$ and $f\in H^1(\Omega)$.
		Remarkably, \eqref{eq:kolmogorov-bounded_dD} is smoothing, despite possessing explicit diffusion in \emph{one spatial dimension only}. Indeed, we have, respectively,
		\[
		[\partial_{x},\mbf{b}\cdot\nabla]=\partial_{y}, \quad 
		[\partial_{y},\mbf{b}\cdot\nabla]=\partial_{z},
		\]
		thereby, H\"ormander's rank condition is satisfied \cite{hormander}, implying that \eqref{eq:kolmogorov-bounded_dD} is, in fact, hypoelliptic! Moreover, since full rank is achieved via commutators involving the skew-symmetric 1st order part  $\mbf{b}\cdot\nabla u$, the PDE in \eqref{eq:kolmogorov-bounded_dD} satisfies also the commutator hypotheses of \cite[Theorem 24]{villani}. The developments discussed in this work can be transferred to this problem also with minor modifications only.

		\end{remark}
	}

Returning to the case of a general polygonal domain $\Omega\subset \mathbb{R}^2$, we shall assume the validity of the additional boundary conditions:
\begin{equation}\label{eq:additional_bc_gen}
\big(\frac{1}{2}xn_2\nabla u-n_1\nabla u_{x}\big)\big|_{\partial\Omega\backslash\partial_+\Omega}=0\quad\text{ and }\quad n_1\nabla u_{x}\big|_{\partial_+\Omega}=0,
\end{equation}
for the remainder of this work, which trivially  imply \eqref{eq:additional_bc}. We do so, instead of prescribing particular additional boundary conditions; we shall clearly state if the validity of \eqref{eq:additional_bc_gen} is assumed or not for the proof of various results below. We stress that, under minor modifications only, the theory presented below can still be valid for other suitable assumptions, \change{e.g., for periodic solutions also} by careful inspection of the proofs. \change{The combination of suitable boundary/closure conditions, whose nature ensures the availability of Poincar\'e-Friedrichs inequalities, is key in the solution behaviour as $t_f\to\infty$. Nonetheless, in the present work, our focus is on the derivation of Galerkin discretizations of the bilinear form $a(\cdot,\cdot)$; as such, we adopt for simplicity the convenient setting \eqref{eq:additional_bc_gen} for hypocoercivity-inducing boundary conditions.} 

\section{A finite element method}\label{sec:fem}
Motivated by the above discussion, we return to the variational representation \eqref{eq:weak_hypo} of problem \eqref{eq:kolmogorov-bounded},  \eqref{eq:additional_bc_gen} on a bounded domain $\Omega\subset \mathbb{R}^2$. For simplicity of the presentation we, henceforth, assume that $\Omega$ is polygonal also; extension to curved domains is possible via isoparametric versions of the finite element spaces discussed below. As is standard in discretisation by a Galerkin type finite element method, we introduce a family of triangulation\change{s} of $\Omega$, say $\mathcal{T}$, consisting of mutually disjoint open triangular elements $T\in \mathcal{T}$, whose closures cover $\bar{\Omega}$ exactly. Let also $h:\cup_{T\in\mathcal{T}} T\to\mathbb{R}_+$ be the local meshsize function defined elementwise by $h|_T:=h_T:={\rm diam}(T)$. For simplicity, we further assume that $\mathcal{T}$ is shape-regular, in the sense that the radius $\rho_T$ of the largest inscribed circle of each $T\in\mathcal{T}$ is bounded \change{from below} with respect to each element's diameter $h|_T$, uniformly as $\|h\|_{L_\infty(\Omega)}\to 0$ under mesh refinement. Also, we assume that $\mathcal{T}$ is locally quasi-uniform in the sense that the diameters of adjacent elements are uniformly bounded from above and below. Finally, let $\Gamma:=\cup_{T\in\mathcal{T}}\partial T$ denote the mesh skeleton, and $\Gamma_{\rm int}:=\Gamma\backslash\partial\Omega$.

We define the \change{standard, continuous,} finite element space subordinate to $\mathcal{T}$ by
\[
V_h\equiv V_h^p:=\{V\in H^1_{-,0}(\Omega): V|_T\in \mathbb{P}_p(T),\ T\in\mathcal{T}\},
\]
with $\mathbb{P}_p(\omega)$, $\omega\subset\mathbb{R}^2$, denoting the space of polynomials of total degree at most $p$ over $\omega$, $p=2,3,\dots$. 
Further, we define the \emph{jump} $\jump{w}$ of $w$ across the common element interface $e:=\partial T^+\cap \partial T^-$ of two adjacent elements $T^+,T^-$, by $\jump{w}|_e:=w|_{T^+}\mathbf{n}|_{T^+}+w|_{T^-}\mathbf{n}|_{T^-}$ with $\mathbf{n}|_T$ denoting the unit outward normal vector to each point of $\partial T$; correspondingly, we also define the \emph{average} $\av{w}|_e:=(w|_{T^+}+w|_{T^-})/2$ of $w$ across $e$. Finally, for boundary faces $e\subset \partial T\cap \partial\Omega$, we set $\jump{w}|_e:=w|_{T}\mathbf{n}|_{T}$ and $\av{w}|_e:=w|_{T}$, respectively. The above jump and average definitions are trivially extended to vector-valued functions by component-wise application. 

We shall also make use of the \emph{broken Sobolev spaces}
$
H^r(\Omega,\mathcal{T}):=\{v|_T\in H^r(T), T\in\mathcal{T}\},
$
with respective Hilbertian norm $\norm{w}{H^r(\Omega,\mathcal{T})}:=\big(\sum_{T\in\mathcal{T}}\norm{\nabla^r w}{T}^2+\norm{\av{h}^{1/2-r}\jump{w}}{\Gamma_{\rm int}}^2\big)^{1/2}$, $r\ge 1$. Further, we define the broken gradient $\nabla_\mathcal{T}$ to be the element-wise gradient operator with $\nabla_{\mathcal{T}}w|_T=\nabla w|_T$, $T\in\mathcal{T}$, for $w\in H^1(\Omega,\mathcal{T})$.

Starting from \eqref{eq:weak_hypo}, we consider the spatially discrete finite element method: \change{find $U\equiv U(t)\in V_h$, $t\in(0,t_f]$,} such that
\begin{equation}\label{eq:FEM}
\begin{aligned}
&(U_t,V)+ (U_x,V_x)+(xU_y,V)\\
+&(\nabla U_t,A\nabla V) +(A\nabla_{\mathcal{T}}  U_x, \nabla_{\mathcal{T}}  V_x)+(\nabla_{\mathcal{T}}(x U_y),A\nabla V)+s_h(U,V)\\
=\ &(f,V)+(\nabla f,A\nabla V),
\end{aligned}
\end{equation}
for all  $V\in V_h$, whereby
\[
\begin{aligned}
s_h(U,V):=&\ -\int_{\Gamma_{\rm int}\cup\partial_-\Omega}(0,x)^T\cdot\big( \alpha\jump{U_x}\av{V_x}+\beta \jump{U_x}\av{V_y}+ \beta \jump{U_y}\av{V_x}+\gamma \jump{U_y}\av{V_y}\big)\ud s\\
&\ +\int_{\Gamma_{\rm int}}\frac{|x n_2|}{2}\big( \kappa\jump{U_x}\cdot\jump{V_x}+\lambda \jump{U_y}\cdot\jump{V_y}\big)\ud s\\
&\ -\int_{\Gamma_{\rm int}\cup\partial_0\Omega}\Big(\av{A\nabla U_x}\cdot[\nabla V]_1+\av{A\nabla V_x}\cdot[\nabla U]_1-\tau [\nabla U]_1\cdot A[\nabla V]_1\Big)\ud s,
\end{aligned}
\]
with $[v]_1|_e:=(v|_e)^+n_1^++(v|_e)^-n_1^-$, for each internal face $e$ and $n_1$ denoting the first component of the unit normal vector to $e$ and $[v]_1|_e:=(v|_e)^+n_1^+$, for $e\subset\partial\Omega$, for some $\tau:\Gamma\to \mathbb{R}$ non-negative function to be determined precisely below and user-defined parameters $\kappa,\lambda\ge 0$;
we also set $U(0):=\Pi u_0\in V_h$, with $\Pi:L_2(\Omega)\to V_h$ denoting the orthogonal $L_2$-projection onto the finite element space. For brevity, we shall adopt the notation:
\[
s_{1,h}(w,v):=\int_{\Gamma_{\rm int}}\frac{|x n_2|}{2}\big( \kappa\jump{w_x}\cdot\jump{v_x}+\lambda \jump{w_y}\cdot\jump{v_y}\big)\ud s.
\]
The term $s_{1,h}(w,v)$ introduces numerical diffusion in the $(0,x)^T$-direction when $\kappa,\lambda>0$. Such additional numerical diffusion may be desirable in the presence of rough initial conditions. When additional numerical diffusion is not advantageous, we can choose $\kappa,\lambda=0$ giving $s_{1,h}(\cdot,\cdot)\equiv 0$.

As we shall see below, the stabilisation $s_h(\cdot,\cdot)$ is constructed so that the \change{method} \eqref{eq:FEM} is \emph{consistent} with \eqref{eq:weak_hypo} upon assuming the additional boundary conditions \eqref{eq:additional_bc_gen}.

\begin{lemma}\label{lem:stab} Let $\tau:\Gamma\to \mathbb{R}$ with $\tau= C_{\tau}p^2/\av{h}$, for $C_{\tau}>0$ large enough but independent of $h$ and of $p$.  Then, for any $0<\alpha < 1/2$, there exist $0< \gamma<\beta<1$ and a positive constant $c_0$, independent of the approximate solution $U\in V_h$ and of $t_f$, such that for any $0<\zeta<1$, we have
	\begin{equation}\label{eq:stab}
	\begin{aligned}
\begin{aligned}
&\|U(t_f)\|^2 +\|\sqrt{A}\nabla U(t_f)\|^2+(1-\zeta)\int_0^{t_f}\big(c_0 \| U\|^2+\|\sqrt{A}\nabla U\|^2\big)\ud t\\
&+\!\int_0^{t_f}\!\!\big(\|\sqrt{xn_2} U\|_{\partial_+\Omega}^2+\|\sqrt{xn_2 A} \nabla U\|_{\partial_+\Omega}^2\\
&\qquad+\|\sqrt{A}\nabla_\mathcal{T} U_x\|^2+ s_{1,h}(U,U)+\|\sqrt{\tau A} [\nabla U]_1\|_{\Gamma_{\rm int}\cup\partial_0\Omega}^2\big)\ud t
\\
\le\  &\frac{4}{\zeta}\int_0^{t_f}\big(c_0^{-1} \|f\|^2 +\|\sqrt{A}\nabla f\|^2\big)\ud t + \|\Pi u_0\|^2 +\|\sqrt{A}\nabla \Pi u_0\|^2;
\end{aligned}
	\end{aligned}
	\end{equation}
	when $f\equiv0$, we can trivially select $\zeta=0$ \change{in the proof, so that the first term on the right-hand side of \eqref{eq:stab} vanishes}. 
\end{lemma}
\begin{proof}
	Setting $V=U$ in \eqref{eq:FEM}, standard arguments give
	\begin{equation}\label{eq:stab_one}
	\begin{aligned}
	&\frac{1}{2}\frac{\ud}{\ud t}\big( \|U\|^2 +\|\sqrt{A}\nabla U\|^2\big) 
	+\|U_x\|^2+\frac{1}{2}\|\sqrt{xn_2} U\|_{\partial_+\Omega}^2\\
&	+\|\sqrt{A}\nabla_{\mathcal{T}} U_x\|^2+(\nabla_{\mathcal{T}}(xU_ y),A\nabla U)+s_h(U,U) \\ 
	=&\ (f,U)+(\nabla f,A\nabla U),
	\end{aligned}
	\end{equation}
	upon observing that $(xU_y,U)=\|\sqrt{xn_2} U\|_{\partial_+\Omega}^2/2$ as $U=0$ on $\partial_-\Omega$. 
We calculate
	\[
	\begin{aligned}
(\nabla(xU_y),A\nabla U)_T =&\ 
	\alpha(U_y,U_x)_T+\beta\|U_y\|_T^2+\alpha(U_x,xU_{yx})_T+\beta(U_y,xU_{yx})_T\\
	&+\beta(xU_{yy},U_x)_T+\gamma(xU_{yy},U_y)_T .
	\end{aligned}
	\]
	Observ\change{e} now the identities
	\[
	\alpha(U_x,xU_{yx})_T = \frac{\alpha}{2}\int_{\partial T}xn_2 U_x^2\ud s,\quad 
	\gamma(U_y,xU_{yy})_T = \frac{\gamma}{2}\int_{\partial T}xn_2 U_y^2\ud s,
	\]
	\[
	\beta(xU_{yy},U_x)_T= -\beta(xU_y,U_{xy})_T +\beta\int_{\partial T}xn_2U_xU_y\ud s\change{.}
	\]
Then, summing over all $T\in\mathcal{T}$, we deduce
	\begin{equation}\label{eq:stab_two}
	\begin{aligned}
(\nabla_{\mathcal{T}}(xU_y),A\nabla U) =&\ \alpha(U_y,U_x)+\beta\|U_y\|^2
	+\frac{1}{2}\int_{\Gamma}(0,x)^T\cdot \jump{\alpha U_x^2+2\beta U_xU_y+\gamma U_y^2}\ud s
	.
	\end{aligned}
	\end{equation}
	Using \eqref{eq:stab_two} into \eqref{eq:stab_one}, noting that $\jump{WV}=\av{W}\jump{V}+\av{V}\jump{W}$ on $\Gamma_{\rm int}$ for $V,W\in V_h$, and integrating with respect to $t\in[0,t_f]$, therefore, gives
	\begin{equation}\label{eq:stab_left}
	\begin{aligned}
	&\|U({t_f})\|^2 +\|\sqrt{A}\nabla U({t_f})\|^2+2\int_0^{t_f}\|\sqrt{A}\nabla_\mathcal{T} U_x\|^2\ud t
	+\int_0^{t_f}\|\sqrt{xn_2} U\|_{\partial_+\Omega}^2\ud t\\
	&+\int_0^{t_f}\int_{\partial_+\Omega}xn_2(\alpha U_x^2+2\beta U_xU_y+\gamma U_y^2)\ud s\ud t
	+2\int_0^{t_f}\big(\|U_x\|^2+\alpha(U_y,U_x)+\beta\|U_y\|^2\big)\ud t\\
	&\ -4\int_0^{t_f}\int_{\Gamma_{\rm int}\cup\partial_0\Omega}\av{A\nabla U_x}[\nabla U]_1\ud s\ud t+2\int_0^{t_f}\big(\|\sqrt{\tau A} [\nabla U]_1\|_{\Gamma_{\rm int}\cup\partial_0\Omega}^2+s_{1,h}(U,U)\big)\ud t\\
	=&\ 2\int_0^{t_f}\big( (f,U)+(\nabla f,A\nabla U)\big)\ud t + \|\Pi u_0\|^2 +\|\sqrt{A}\nabla \Pi u_0\|^2.
	\end{aligned}
	\end{equation}
	Focusing on the fifth and sixth terms on the left-hand side of \eqref{eq:stab_left}, recalling \eqref{eq:bigA}, 
	the Cauchy-Schwarz and Young inequalities, along with elementary manipulations, imply
		\begin{equation}\label{eq:stab_left0}
	\int_{\partial_+\Omega}\!\!xn_2(\alpha U_x^2+2\beta U_xU_y+\gamma U_y^2)\ud s=	\int_{\partial_+\Omega}\!\!xn_2 |\sqrt{A}\nabla U|^2 \ud s
	\ge 0,
	\end{equation}
	and
	\begin{equation}\label{eq:stab_left1}
2\big(\|U_x\|^2+\alpha(U_y,U_x)+\beta\|U_y\|^2\big)
	\ge 2(1-\epsilon)\|U_x\|^2+\big(2\beta-\frac{\alpha^2}{2\epsilon}\big)\|U_y\|^2=\|\sqrt{B}\nabla U\|^2,
	\end{equation}
respectively, upon defining the diagonal matrix $B={\rm diag}(2(1-\epsilon),2\beta-\alpha^2/(2\epsilon))$, for any $0<\epsilon<1$.
	We, then, have
	\[
	\|\sqrt{B}\nabla U\|^2= \|\sqrt{B-A}\nabla U\|^2+\|\sqrt{A}\nabla U\|^2
	\ge 
	\lambda_{B-A}^{\min} \|\nabla U\|^2+\|\sqrt{A}\nabla U\|^2,
	\]
	where $\lambda_{B-A}^{\min}$ denotes the smallest eigenvalue of $B-A$. There exist values $0<\gamma<\beta$ such that $\lambda_{B-A}^{\min}>0$. For instance, setting $ \beta = \alpha^2$, $\gamma=\alpha^3$, and $\epsilon=1/2$, we have ${\rm det}(B-A) = \alpha^2(1-2\alpha)$.
	A numerical investigation reveals that $\max_{0\le \alpha\le 1/2}\lambda_{B-A}^{\min}\approx 0.054429$ attained for a value $\alpha\approx 0.35060$, noting that $B-A$ has positive determinant for $0<\alpha<1/2$.

	 Since $U=0$ on $\partial_{-,0}\Omega$, which was assumed to be of positive one-dimensional Hausdorff measure, the validity of a Poincar\'e-Friedrichs/spectral gap inequality $\|U\|^2\le C_{PF}\|\nabla U\|^2$ with $C_{PF}\equiv C_{PF}(\Omega)$ positive constant, independent of $U$, is assured, thereby allowing us to conclude the lower bound
	\begin{equation}\label{eq:PF}
	\|\sqrt{B}\nabla U\|^2
	\ge 
	C_{PF}^{-1}\lambda_{B-A}^{\min} \| U\|^2+\|\sqrt{A}\nabla U\|^2.
	\end{equation}
	
 Thus, with the choices $\alpha\approx 0.35060$, $ \beta = \alpha^2$, $\gamma=\alpha^3$, and $\epsilon=1/2$, we can select $c_0= 0.054429/C_{PF}$. As we shall see below, $c_0$ effectively represents the exponent of the exponential decay to equilibrium\change{;} cf. Theorem \ref{thm:hypoco} below.

	Further, the symmetry of $A$ along with the Cauchy-Schwarz inequality and a standard inverse estimate yield, respectively,
	\begin{equation}\label{eq:dg_forever}
	\begin{aligned}
	\int_{\Gamma_{\rm int}\cup\partial_0\Omega}\av{A\nabla U_x}\cdot [\nabla U]_1\ud s&\le \|\varphi^{-1/2}\av{\sqrt{A}\nabla U_{x}}\|_{\Gamma_{\rm int}\cup\partial_0\Omega} \| \sqrt{\varphi A}[\nabla U]_1\|_{\Gamma_{\rm int}\cup\partial_0\Omega}\\
	&\le \|\sqrt{C_{\rm inv}p^2/h \varphi}\sqrt{A}\nabla_{\mathcal{T}} U_{x}\| \| \sqrt{\varphi A}[\nabla U]_1\|_{\Gamma_{\rm int}\cup\partial_0\Omega}\\
	&\le \frac{1}{4} \|\sqrt{A}\nabla_{\mathcal{T}}  U_{x}\|^2 +\|\sqrt{\tau A} [\nabla U]_1\|_{\Gamma_{\rm int}\cup\partial_0\Omega}^2\change{,}
	\end{aligned}
	\end{equation}
	 using the standard inverse inequality $\|v\|_{\partial{T}}^2\le C_{\rm inv}p^2/h_T\|v\|^2_{T}$ on \eqref{eq:dg_forever}, and selecting $\varphi=\tau\ge 4C_{\rm inv}p^2/\av{h}$ and resorting to the local quasiuniformity of the mesh.
	Inserting the last two bounds into \eqref{eq:stab_left} and, using the Cauchy-Schwarz inequality, setting $c_0:=C_{PF}^{-1}\lambda_{B-A}^{\min} $, we deduce the result
	for any $0< \zeta<1$; when $f\equiv0$ we can trivially select $\zeta=0$.
\end{proof}

We note that the specific choice of $A$ as a function of $\alpha$ provided in the proof above is \emph{not} unique. As such, we have avoided providing exact values of $\alpha,\beta,\gamma$ in the statement of Lemma \ref{lem:stab}. Nonetheless, preliminary numerical investigation shows that the choice of $A$ does not affect significantly the performance of the method, provided that $A$ is chosen to ensure \eqref{eq:PF} away from limiting values \cite{dg}.

	\begin{remark}  It is possible to optimise further the constant $c_0$ by working as follows: starting from $\|\sqrt{B}\nabla U\|^2$, we have
	\[	\|\sqrt{B}\nabla U\|^2= \|\sqrt{B-\nu A}\nabla U\|^2+\nu\|\sqrt{A}\nabla U\|^2
	\ge 
	\lambda_{B-\nu A}^{\min} \|\nabla U\|^2+\nu\|\sqrt{A}\nabla U\|^2,\]
for suitable $\nu>0$. This way it is possible to arrive at a different, possibly larger, value for $\lambda_{B-\nu A}^{\min}$ and, consequently, for $c_0$.  
	\end{remark}

The last proof highlights that the validity of the Poincar\'e-Friedrichs inequality is of central importance here, and in general in the study of trend to equilibrium for kinetic equations \cite{risken,herau_nier,villani}. In the present context of bounded spatial domain $\Omega$, we have taken the viewpoint of using standard (non-weighted) integral norms. Thus the Lebesgue measure of $\Omega$ is present in the constant $c_0$. The extension of the above framework to general Fokker-Planck equations on function spaces weighted by equilibrium distributions is an important question and will be discussed elsewhere. The availability of such a Poincar\'e-Friedrichs inequality (also known as spectral gap property in kinetic theory) facilitates the crucial feature of the proposed method: the absence of an exponential term of the form $\exp(ct_f)$ multiplying the data terms in the above stability estimate. Moreover, \eqref{eq:FEM} also immediately implies the decay of numerical solutions for the respective homogeneous problem. 

\begin{theorem}[Decay via hypocoercivity]\label{thm:hypoco}
Under the hypotheses of Lemma \ref{lem:stab}, the  finite element method \eqref{eq:FEM} for the homogeneous problem \eqref{eq:kolmogorov-bounded} with $f\equiv 0$ satisfies
	\[
	\|U(t_f)\|^{\change{2}} +\|\sqrt{A}\nabla U(t_f)\|^{\change{2}}\le {\rm e}^{-\min\{1,c_0\}t_f}\Big(\|\Pi u_0\|^{\change{2}} +\|\sqrt{A}\nabla \Pi u_0\|^{\change{2}}\Big),
	\]
	for all $t_f>0$.
\end{theorem}
\begin{proof}
\change{Starting from \eqref{eq:stab_one} and using the estimates from the proof of Lemma \ref{lem:stab}, we arrive at the estimate}
	\[
\frac{\ud}{\ud t}\big( \|U\|^2 +\|\sqrt{A}\nabla U\|^2\big) 
	+\min\{1,c_0\}\big( \| U\|^2+\|\sqrt{A}\nabla U\|^2\big)
	\le 0,
	\]
	from which an application of Gr\"onwall's Lemma already implies the result. 	
\end{proof}

\section{Error analysis}\label{sec:error_analysis}
We continue by proving a priori error bounds for \eqref{eq:FEM}. For brevity, we shall denote by $\mathcal{B}(\cdot,\cdot):H^{5/2+\varepsilon}(\Omega,\mathcal{T})\times H^{5/2+\varepsilon}(\Omega,\mathcal{T})\to \mathbb{R}$, $\varepsilon>0$, the bilinear form given by
\[
\mathcal{B}(w,v):=(w_x,v_x)+(xw_y,v) +(A\nabla_{\mathcal{T}}  w_x, \nabla_{\mathcal{T}}  v_x)+(\nabla_{\mathcal{T}}(x w_y),A\nabla v)+s_h(w,v).
\]

Implicitly in the proof of Lemma \ref{lem:stab}, we proved also the following (hypo)coercivity result.
\begin{lemma}[(Hypo)coercivity]\label{lem:coercivity} Let
	\[
	\begin{aligned}
	\tnorm{w}:= &\Big(\|\sqrt{B}\nabla w\|^2  
	+\|\sqrt{xn_2} w\|_{\partial_+\Omega}^2+\|\sqrt{xn_2 A} \nabla w\|_{\partial_+\Omega}^2\\
	&	+ \|\sqrt{A}\nabla_\mathcal{T} w_x\|^2
	+\|\sqrt{\tau A} [\nabla w]_1\|_{\Gamma_{\rm int}\cup\partial_0\Omega}^2+ s_{1,h}(w,w)\Big)^{1/2},
	\end{aligned}
	\]
	with $B$ as in the proof of Lemma \ref{lem:stab}. Then, $\tnorm{\cdot}$ is a norm on $H^1_{-,0}(\Omega)\cap H^2(\Omega,\mathcal{T})$, and we have
	\[
\mathcal{B}(w,w)\ge\frac{1}{2}\tnorm{w}^2,
	\]
	for all $w\in V_h$.
	\qed
\end{lemma}
 
 For the remaining of this work, unless explicitly stated, we assume that $\epsilon$ is away from $1$ and that $\beta\sim \alpha^2$, so that $B=c_{B}{\rm diag}(1,\alpha^2)$ for a given constant $0<c_B\le 1$. Thus, we can take $C_{\rm 0}= (2c_B)^{-1}$. Note that this also covers the case of $A$ being singular, i.e., the operator $\nabla\cdot A\nabla(\cdot)$ being parabolic.  Next, we establish the consistency of the bilinear form $\mathcal{B}$.
\begin{lemma}\label{lem:incons}
Assume that for the solution $u$ of \eqref{eq:kolmogorov-bounded}, \eqref{eq:additional_bc_gen} we have $u(t,\cdot)\in H^1_{-,0}(\Omega)\cap H^{5/2+\varepsilon}(\Omega,\mathcal{T})$, $\varepsilon>0$, for almost all $t\in(0,t_f]$. Then, for almost all $t\in(0,t_f]$ and for all $V\in V_h$, we have
	\begin{equation}\label{eq:incons}
	(u_t,V)+(\nabla u_t,A\nabla V)+\mathcal{B}(u,V) = (f,V)+(\nabla f,A\nabla V).
	\end{equation}
	\end{lemma}
\begin{proof}
	The proof follows by integration by parts and by the smoothness of $u$, noting that from \eqref{eq:additional_bc_gen}, we have
$
	(\frac{1}{2}xn_2\nabla u-n_1\nabla_{\mathcal{T}} u_{x}, A\nabla V)_{\partial\Omega\backslash\partial_+\Omega}=0$ and $(n_1\nabla{\mathcal{T}} u_{x}, A\nabla v)_{\partial_+\Omega}=0.
	$\end{proof}

Let $\pi:H^1_{-,0}(\Omega)\to V_h$, a projection operator onto the finite element space to be defined precisely below, and set 
\[
\rho:=u-\pi u,\qquad \vartheta:=\pi u -U.
\]
Employing \eqref{eq:incons} and \eqref{eq:FEM} with $V=\vartheta$, gives the error equation
\begin{equation}\label{eq:error_eq}
\begin{aligned}
&\frac{1}{2}\frac{\ud }{\ud t}\big(\|\vartheta(s)\|^2+\|\sqrt{A}\nabla \vartheta(s)\|^2\big)+\mathcal{B}(\vartheta,\vartheta)
= -(\rho_t,\vartheta)-(\nabla \rho_t,A\nabla \vartheta)-\mathcal{B}(\rho,\vartheta).
\end{aligned}
\end{equation}
We shall now construct a $\pi$, so that $\mathcal{B}(\rho,\vartheta)=0$. 

\begin{lemma}[(\change{H}ypo)elliptic projection]\label{lem:hypo_proj} Assume that $B=c_{B}{\rm diag}(1,\alpha^2)$, for a given constant $0<c_B\le 1$.
	For every $v\in H^1_{-,0}(\Omega)\cap H^{5/2+\varepsilon}(\Omega,\mathcal{T})$, $\varepsilon>0$, there exists a unique $\pi v\in V_h$ such that
	\begin{equation}\label{eq:hypo_proj}
	\begin{aligned}
	\mathcal{B}(\pi v,V)
	= \mathcal{B}(v,V), 
	\end{aligned}
	\end{equation}
	for all $V\in V_h$. Moreover, assuming further that $v\in H^{k_T}(T)$ for $T\in\mathcal{T}$, $k_T\ge3$, we have the approximation estimate:
	\begin{equation}\label{eq:error_bound_en}
	\tnorm{v-\pi v}^2\le C(A,\kappa,\lambda)\sum_{T\in\mathcal{T}}\frac{h_T^{2s_T}}{p^{2(k_T-3)}}x_T|u|_{H^{k_T}(T)}^2
,
	\end{equation}
	with $s_T:={\min\{p+1,k_T\}-2}$ and $x_T:=\max_{x\in T} \{1,|x|^2\}$, $T\in\mathcal{T}$.
\end{lemma}
\begin{proof}
Uniqueness (and, therefore, existence due to linearity) follows from the coercivity of $\change{\mathcal{B}}$ shown in Lemma \ref{lem:coercivity}. Setting now $\eta:=v-\mathcal{P} v$ and $\xi:=\mathcal{P} v-\pi v$, with $\mathcal{P}$ denoting a projection operator from $H^1_{-,0}(\Omega)$ onto $V_h$ \change{with optimal approximation properties with respect to the local meshsize $h$}, we have
\begin{equation}\label{eq:apriori_xi}
\begin{aligned}
\frac{1}{2}\tnorm{\xi}^2
\le&\ \mathcal{B}(\xi, \xi) = -\mathcal{B}(\eta, \xi )\\
=& -(\eta_x,\xi_x)-(\eta_y,x\xi) -(\sqrt{A}\nabla_{\mathcal{T}}  \eta_x, \sqrt{A}\nabla_{\mathcal{T}}  \xi_x)\\
&-(\nabla_{\mathcal{T}}(x \eta_y),A\nabla \xi)-s_h(\eta,\xi).
\end{aligned}
\end{equation}
Now, integration by parts yields
\begin{equation}\label{eq:apriori_one}
-(\eta_y,x\xi) =(x\eta,\xi_y) -\int_{\partial_+\Omega}xn_2\eta\xi\ud s,
\end{equation}
and
\begin{equation}\label{eq:apriori_two}
\begin{aligned}
&(\nabla_{\mathcal{T}}(x \eta_y),A\nabla \xi) \\
=& \sum_{T\in\mathcal{T}} \Big( (\eta_y + x\eta_{xy}, \alpha\xi_x+\beta\xi_y)_T+(x\eta_{yy} , \beta\xi_x+\gamma\xi_y)_T\Big)\\
=&\ \alpha(\eta_y,\xi_x)+\beta(\eta_y,\xi_y)
- \sum_{T\in\mathcal{T}} \Big( (\eta_{x},x(\alpha\xi_{x}+\beta\xi_y)_y)_T +(\eta_{y},x(\beta\xi_{x}+\gamma\xi_y)_y)_T\Big)\\
&+ \sum_{T\in\mathcal{T}} \int_{\partial T} x n_2\big( \alpha \eta_{x}\xi_{x}+\beta\eta_x\xi_y+\beta\eta_y\xi_{x}+\gamma\eta_y\xi_y\big)\ud s.
\end{aligned}
\end{equation}
Observe that the last term on the right-hand side of \eqref{eq:apriori_two} is equal to
\begin{equation}\label{eq:apriori_three}
\begin{aligned}
 &\int_{\Gamma_{\rm int}\cup\partial_-\Omega}(0,x)^T\cdot \big( \alpha \jump{\eta_{x}}\av{\xi_{x}}+\beta\jump{\eta_x}\av{\xi_y}+\beta\jump{\eta_y}\av{\xi_{x}}+\gamma\jump{\eta_y}\av{\xi_y}\big)\ud s\\
 &+\int_{\Gamma_{\rm int}}(0,x)^T\cdot \big( 
 \alpha \av{\eta_{x}}\jump{\xi_{x}}+\beta\av{\eta_x}\jump{\xi_y}+\beta\av{\eta_y}\jump{\xi_{x}}+\gamma\av{\eta_y}\jump{\xi_y}\big)\ud s\\
 &+\int_{\partial_+\Omega}xn_2 A\nabla \eta \cdot \nabla \xi\ud s,
\end{aligned}
\end{equation}
and note that the first term of \eqref{eq:apriori_three} cancels with the first term of $s_h(\eta,\xi)$.
Using now \eqref{eq:apriori_one}, \eqref{eq:apriori_two} and \eqref{eq:apriori_three} into \eqref{eq:apriori_xi}, along with standard Cauchy-Schwarz' and Young's inequality arguments, and working as in \eqref{eq:dg_forever} gives
\begin{equation*}
\begin{aligned}
&C^{-1}\tnorm{\xi}^2\\
\le&\ \norm{\eta_x}{}^2+ \beta^{-1}\norm{x\eta}{}^2
+\norm{\sqrt{xn_2}\eta}{\partial_+\Omega}^2
+\norm{\sqrt{xn_2A}\nabla\eta}{\partial_+\Omega}^2+\norm{\sqrt{A}\nabla_{\mathcal{T}}  \eta_x}{}^2\\
&+(\alpha^2+\beta) \norm{\eta_y}{}^2+\norm{\sqrt{\tau A}[\nabla \eta]_1}{\Gamma_{\rm int}}^2+\kappa\norm{\sqrt{|xn_2|}\jump{\eta_x}}{\Gamma_{\rm int}}^2+\lambda\norm{\sqrt{|xn_2|}\jump{\eta_y}}{\Gamma_{\rm int}}^2\\
&+ \sum_{T\in\mathcal{T}}\Big( \norm{\mu^{-1}x\eta_{x}}{T}\big(\alpha\norm{\mu(\xi_{x})_y}{T}
+\beta\norm{\mu(\xi_y)_y}{T}\big)
+\norm{\mu^{-1}x\eta_{y}}{T}\big(\beta\norm{\mu(\xi_{x})_y}{T}
+\gamma\norm{\mu(\xi_y)_y}{T}\big)\Big)\\
&+\norm{\sqrt{A/\tau}\av{\nabla\eta_x}}{\Gamma_{\rm int}}^2 +
\big(\alpha^2\kappa^{-1}+\beta^2\lambda^{-1}\big) \norm{\sqrt{xn_2}\av{\eta_{x}}}{\Gamma_{\rm int}}^2
+\big(\beta^2\kappa^{-1}+\gamma^2\lambda^{-1}\big) \norm{\sqrt{xn_2}\av{\eta_{y}}}{\Gamma_{\rm int}}^2
,
\end{aligned}
 \end{equation*}
for any $\mu,\kappa,\lambda>0$, for some $C>1$, independent of the, relevant to the argument, parameters. Alternatively, when $\kappa,\lambda=0$, we work as follows: using the inverse estimate $\norm{\xi_x}{\partial T}^2\le Cp^2/h_T\norm{\xi_x}{T}^2$ and, similarly for $\xi_y$, we further estimate the second term on the right-hand side of \eqref{eq:apriori_three} to arrive instead at
\begin{equation*}
\begin{aligned}
&C^{-1}\tnorm{\xi}^2\\
\le&\ \norm{\eta_x}{}^2+ \beta^{-1}\norm{x\eta}{}^2
+\norm{\sqrt{xn_2}\eta}{\partial_+\Omega}^2+\norm{\sqrt{A}\nabla_{\mathcal{T}}  \eta_x}{}^2+(\alpha^2+\beta) \norm{\eta_y}{}^2+\norm{\sqrt{\tau A}[\nabla \eta]_1}{\Gamma}^2\\
&+ \sum_{T\in\mathcal{T}}\Big( \norm{\mu^{-1}x\eta_{x}}{T}\big(\alpha\norm{\mu(\xi_{x})_y}{T}
+\beta\norm{\mu(\xi_y)_y}{T}\big)
+\norm{\mu^{-1}x\eta_{y}}{T}\big(\beta\norm{\mu(\xi_{x})_y}{T}
+\gamma\norm{\mu(\xi_y)_y}{T}\big)\Big)\\
&+
C|AB^{-1/2}|_2\big( \norm{xn_2h^{-1/2}\av{\eta_{x}}}{\Gamma_{\rm int}}^2+\norm{xn_2h^{-1/2}\av{\eta_{y}}}{\Gamma_{\rm int}}^2\big)
+\frac{1}{4}\norm{\sqrt{B}\nabla\xi}{}^2,
\end{aligned}
\end{equation*}
with $|\cdot|_2$ denoting the matrix-2-norm. Applying the inverse estimate $\norm{\nabla v}{T}^2\le Cp^4/h^2_T\norm{v}{T}^{\change{2}}$, for $v\in \mathbb{P}_p(T)$ on the terms containing $\mu\xi$ in the last two alternative estimates, selecting $\mu=h/p^2$, and performing standard manipulations, we arrive at the combined estimate
\begin{equation*}
\begin{aligned}
C^{-1}\tnorm{\xi}^2
\le&\ \norm{\eta_x}{}^2+ \beta^{-1}\norm{x\eta}{}^2
+\norm{\sqrt{xn_2}\eta}{\partial_+\Omega}^2+\norm{\sqrt{A}\nabla_{\mathcal{T}}  \eta_x}{}^2+\norm{\sqrt{\tau A}[\nabla \eta]_1}{\Gamma}^2 \\
&+  (\alpha^2+\beta)\big(\norm{\mu^{-1}x\eta_{x}}{}^2+\norm{\eta_y}{}^2\big)
+\big(\beta^2+\frac{\gamma^2}{\beta}\big)\norm{\mu^{-1}x\eta_{y}}{}^2\\
&+\min\Big\{
\big(\alpha^2\kappa^{-1}+\beta^2\lambda^{-1}\big) \norm{\sqrt{xn_2}\av{\eta_{x}}}{\Gamma_{\rm int}}^2
+\big(\beta^2\kappa^{-1}+\gamma^2\lambda^{-1}\big) \norm{\sqrt{xn_2}\av{\eta_{y}}}{\Gamma_{\rm int}}^2,\\
&\qquad\qquad
C|AB^{-1/2}|_2\big( \norm{xn_2h^{-1/2}\av{\eta_{x}}}{\Gamma_{\rm int}}^2+\norm{xn_2h^{-1/2}\av{\eta_{y}}}{\Gamma_{\rm int}}^2\big)
\Big\}\\
&+\kappa\norm{\sqrt{|xn_2|}\jump{\eta_x}}{\Gamma_{\rm int}}^2+\lambda\norm{\sqrt{|xn_2|}\jump{\eta_y}}{\Gamma_{\rm int}}^2,
\end{aligned}
\end{equation*}
which holds for any $\kappa,\lambda\ge 0$.
Choosing $\mathcal{P}$ to be an $hp$-optimal projection operator onto the finite element space $V_h$, we can assume approximation estimates of the form
\[
\norm{\nabla^m\eta}{T}\le Ch_T^{\min\{p+1,k\}-m}p^{m-k}|u|_{H^{k}(T)}
\]
for $k>m$, with $u\in H^k(T)$, $T\in\mathcal{T}$ (see, e.g., \cite{schwab,cdgh_book} for examples of such operators). These, together with the trace estimate $\norm{w}{\partial T}^2\le C\norm{w}{T}\norm{w}{H^1(T)}$ and the triangle inequality 
$\tnorm{\rho}\le\tnorm{\eta}+\tnorm{\xi}$, already imply \eqref{eq:error_bound_en}.
\end{proof}

We now show the (potentially super-)approximation for the left-hand side of \eqref{eq:error_eq}. Such results are typical to finite element method for parabolic problems. We shall show that this is the case also for \eqref{eq:FEM} discretising the, degenerate, Kolmogorov equation.

\begin{lemma} Assume that the exact solution $u$ to  \eqref{eq:kolmogorov-bounded}, \eqref{eq:additional_bc_gen} satisfies $u\in H^1(0,t_f;H^1_{-,0}(\Omega)\cap H^{5/2+\varepsilon}(\Omega,\mathcal{T})) $. Then, under
 the hypotheses of Lemma \ref{lem:stab}, we have the following bound
\begin{equation}\label{eq:bound4theta}
\begin{aligned}
\norm{\vartheta(t)}{}^2+\norm{\sqrt{A}\nabla\vartheta(t)}{}^2+\frac{1}{2}\int_0^t\tnorm{\vartheta}{}^2\ud s
\le&\  \norm{\vartheta(0)}{}^2+\norm{\sqrt{A}\nabla\vartheta(0)}{}^2\\
&+4\max\{1,c_0^{-2}\}\int_0^t\norm{\sqrt{B}\nabla\rho_t}{}^2\ud s
,
\end{aligned}
\end{equation}	
for any $t\in(0,t_f]$, for some positive constant $C$, depending only on the shape-regularity and the local quasi-uniformity of the mesh.
\end{lemma}

\begin{proof} Starting from \eqref{eq:error_eq}, upon 
	integration in time between $0$ and $t\in(0,t_f]$, multiplying by $2$ and using standard arguments, gives
	\begin{equation}\label{eq:bound4theta2}
	\begin{aligned}
\norm{\vartheta(t)}{}^2+\norm{\sqrt{A}\nabla\vartheta(t)}{}^2+\int_0^t\tnorm{\vartheta}{}^2\ud s\le&\  \norm{\vartheta(0)}{}^2+\norm{\sqrt{A}\nabla\vartheta(0)}{}^2\\
&+2\int_0^t\big(c_0^{-1}\norm{\rho_t}{}^2 +\norm{\sqrt{A}\nabla\rho_t}{}^2\big)\ud s\\
&+ \frac{1}{2}\int_0^t \big(c_0\norm{\vartheta}{}^2 +\norm{\sqrt{A}\nabla\vartheta}{}^2\big)\ud s,
\end{aligned}
	\end{equation}
	recalling the notation $c_0=C_{PF}^{-1}\lambda^{\min}_{B-A}$. Using now \eqref{eq:PF}, we deduce the result.
\end{proof}

The above developments yield an a priori error bound in the norm appearing on the left-hand side of \eqref{eq:bound4theta}.

\begin{theorem}\label{thm:error_analysis_one} Assume that $u_0, u, u_t\in H^1_{-,0}(\Omega)\cap H^{k_T}(T)$ for $k_T\ge 3$, $T\in\mathcal{T}$, the latter two for almost all $t\in(0,t_f]$, with $u$ being the exact solution to   \eqref{eq:kolmogorov-bounded}, \eqref{eq:additional_bc_gen}.
	Then, under the assumptions of Lemmata \ref{lem:stab} and \ref{lem:hypo_proj}, the error $e:=u-U$ of the finite element method \eqref{eq:FEM} satisfies the bound:
		\begin{equation}\label{eq:apriori_bound}
	\begin{aligned}
	\norm{e(t)}{}^2+\norm{\sqrt{A}\nabla e(t)}{}^2+\int_0^t\tnorm{e}{}^2\ud s\le&\  C(A,\kappa,\lambda)
	\sum_{T\in\mathcal{T}} \mathcal{E}_T(u_0,t,u)
	,
	\end{aligned}
	\end{equation}
	where
	\[
\mathcal{E}_T(u_0,t,u):=	\frac{h_T^{2s_T}x_T}{p^{2(k_T-3)}} \Big(\max\{1,c_0^{-1}\}|u_0|_{H^{k_T}(T)}^2 + |u_t|_{L_2(0,t;H^{k_T}(T))}^2+|u|_{L_2(0,t;H^{k_T}(T))}^2\Big),
	\]
with $s_T:={\min\{p+1,k_T\}-2}$, and the constant $C(A,\kappa,\lambda)>0$ is bounded away from $+\infty$ for $0< \alpha\le 1/2$ and all $\kappa,\lambda\ge 0$ and independent of the mesh parameters and of $u$. Moreover, $C(A,\kappa,\lambda)$ is \emph{independent} of the final time $t_f$.
\end{theorem}
\begin{proof}
 Working as for \eqref{eq:PF}, we have
 \[
 \begin{aligned}
 \norm{\vartheta(0)}{}^2+\norm{\sqrt{A}\nabla\vartheta(0)}{}^2\le&\  \max\{1,c_0^{-1}\} \norm{\sqrt{B}\nabla\vartheta(0)}{}^2 \\
 \le &\ 2 \max\{1,c_0^{-1}\}\big( \norm{\sqrt{B}\nabla\rho(0)}{}^2 +\norm{\sqrt{B}\nabla(u_0-\Pi u_0)}{}^2 \big).
 \end{aligned}
 \]
 Similarly, we also have
 \[
 \begin{aligned}
 \norm{\rho(0)}{}^2+\norm{\sqrt{A}\nabla\rho(0)}{}^2\le&\  \max\{1,c_0^{-1}\} \norm{\sqrt{B}\nabla\rho(0)}{}^2.
 \end{aligned}
 \]
 The last two estimates can be further bounded from above by \eqref{eq:error_bound_en} and $hp$-approximation bounds for the orthogonal $L_2$-projection \cite{li_shen,chernov}, respectively.
 Also, from \eqref{eq:error_bound_en}, we have
 \begin{equation}\label{eq:rho_t_estimate}
\int_0^t\big(\norm{\sqrt{B}\nabla\rho_t}{}^2+\tnorm{\rho}{}^2\big)\ud s\le C\sum_{T\in\mathcal{T}}\frac{h_T^{2s_T}x_T}{p^{2(k_T-3)}}\big(|u_t|_{L_2(0,t;H^{k_T}(T))}^2+|u|_{L_2(0,t;H^{k_T}(T))}^2\big).
 \end{equation}
Combining the above bounds with the triangle inequality recalling that $u-u_h=\rho+\vartheta$, the result already follows.
\end{proof}
The crucial property of the above error estimate is the \emph{independence} of the respective constant $C(A,\kappa,\lambda)$ from the final time $t_f$; this is a direct consequence of the hypocoercivity-compatible discretisation \eqref{eq:FEM}.

Assumption $u_0\in H^1_{-,0}(\Omega)$ is a natural compatibility condition in this context. Since a typical setting in kinetic simulation concerns initial profiles with compact support within a computational domain $\Omega$, $u_0\in H^1_{-,0}(\Omega)$ is of significant practical relevance also. Nonetheless, we shall now discuss the dependence of the error committed by the numerical method as $t_f\to\infty$ and we shall see that the effect of the initial condition error $u_0-\Pi u_0$ diminishes exponentially as $t_f\to\infty$. Thus, possibly incompatible initial conditions $u_0\in H^1(\Omega)\backslash H^1_{-,0}(\Omega)$ will have an exponentially diminishing effect in the accuracy of the method with respect to $t_f$.

\begin{theorem}\label{thm:error_analysis_two} \change{Assume that $u, u_t\in H^1_{-,0}(\Omega)\cap H^{k_T}(T)$ for $k_T\ge 3$, $T\in\mathcal{T}$  for almost all $t\in(0,t_f]$, with $u$ being the exact solution to   \eqref{eq:kolmogorov-bounded}, \eqref{eq:additional_bc_gen}, and $u_0\in H^1(\Omega)$.}
	Under the hypotheses of Lemmata \ref{lem:stab} and \ref{lem:hypo_proj} and of Theorem \ref{thm:error_analysis_one}, the error $e:=u-U$ of the finite element method \eqref{eq:FEM} satisfies the bound:
	\begin{equation}\label{eq:apriori_bound_hypoco}
	\begin{aligned}
	\norm{e(t_f)}{}^2+\norm{\sqrt{A}\nabla e(t_f)}{}^2\le  &\ {\rm e}^{-\min\{1,c_0\}t_f} \big(\|\pi u_0-\Pi u_0\|^2+\|\sqrt{A}\nabla (\pi u_0-\Pi u_0)\|^2 \big)\\
	&+\frac{C(A,\kappa,\lambda)}{\min\{1,c_0\}}\sum_{T\in\mathcal{T}}	\frac{h_T^{2s_T}x_T}{p^{2(k_T-3)}}\mathcal{H}_T(t,u)
	,
	\end{aligned}
	\end{equation}
	where
	\[
\mathcal{H}_T(t,u):=|u(t_f)|_{H^{k_T}(T)}^2+\int_0^ {t_f}{\rm e}^{-\min\{1,c_0\}(t_f-t)}  |u_t|_{H^{k_T}(T)}^2 \ud t,
	\]
for $s_T:={\min\{p+1,k_T\}-2}$, with $C(A,\kappa,\lambda)$ \emph{independent} of $t_f$, of $u$, and of the mesh parameters. 
\end{theorem}
\begin{proof}
Starting from \eqref{eq:error_eq}, we use Lemmata \ref{lem:coercivity} and \ref{lem:hypo_proj} and, subsequently, apply \eqref{eq:PF} to arrive at
\[
\begin{aligned}
&\frac{\ud }{\ud t}\big(\|\vartheta(t)\|^2+\|\sqrt{A}\nabla \vartheta(t)\|^2\big)+\min\{1,c_0\}\big(\|\vartheta\|^2+\|\sqrt{A}\nabla \vartheta\|^2\big)
\le c_0^{-1}\norm{\rho_t}{}^2+\norm{\sqrt{A}\nabla \rho_t}{}^2,
\end{aligned}
\]
upon multiplication by $2$. Gr\"onwall's Lemma, thus, implies
\begin{equation}\label{eq:decay_error}
\begin{aligned}
\|\vartheta(t_f)\|^2+\|\sqrt{A}\nabla \vartheta(t_f)\|^2 \le &\ {\rm e}^{-\min\{1,c_0\}t_f} \big(\|\vartheta(0)\|^2+\|\sqrt{A}\nabla \vartheta(0)\|^2 \big)\\
&+\int_0^ {t_f}{\rm e}^{-\min\{1,c_0\}(t_f-t)}\big(c_0^{-1}\norm{\rho_t}{}^2+\norm{\sqrt{A}\nabla \rho_t}{}^2\big)\ud t.
\end{aligned}
\end{equation}
As before, we also have
\[
\begin{aligned}
 \norm{\rho_t}{}^2+\norm{\sqrt{A}\nabla\rho_t}{}^2
 \le&\ \max\{1,c_0^{-1}\} \norm{\sqrt{B}\nabla\rho_t}{}^2
 \le \frac{C(A,\kappa,\lambda)}{\min\{1,c_0\}} \sum_{T\in\mathcal{T}}	\frac{h_T^{2s_T}x_T}{p^{2(k_T-3)}}|u_t|_{H^{k_T}(T)}^2.
\end{aligned}
\]
 The last bound combined with \eqref{eq:decay_error} and the completely analogous estimate 
\[
\|\rho(t_f)\|^2+\|\sqrt{A}\nabla \rho(t_f)\|^2\le \max\{1,c_0^{-1}\} \norm{\sqrt{B}\nabla\rho(t_f)}{}^2,
\]
already imply the result.
\end{proof}
	
Therefore, \eqref{eq:FEM} admits completely analogous long time properties compared to the PDE problem \eqref{eq:kolmogorov-bounded}, \eqref{eq:additional_bc_gen}.

\begin{remark}
The first term on the right-hand side of \eqref{eq:decay_error} vanishes upon altering the finite element method \eqref{eq:FEM} at $t=0$ from $U(0)=\Pi u_0$ to  $U(0)=\pi u_0$. In practical terms, this results in a computational overhead of a stiffness matrix solve as opposed to a mass matrix one. 
\end{remark}

The a priori error estimate derived in Theorem \ref{thm:error_analysis_one} is optimal with respect to mesh-size $h$, upon observing that the bilinear form $\change{\mathcal{B}}$ is the weak form of a formally 4th order differential operator. At the same time, the highest order terms involving $A$ of $\change{\mathcal{B}}(w,v)$ play the role of ``stabilisation'' in this non-standard Galerkin context, resulting in increase of the spatial operator order from second to fourth. The estimate \eqref{eq:apriori_bound} is slightly suboptimal with respect to the polynomial degree $p$, as is typical in error analyses of $hp$-version interior penalty procedures, such as \eqref{eq:FEM}, involving inverse estimates  \cite{ghm}. Correspondingly, the a priori error bound in Theorem \ref{thm:error_analysis_two} is formally suboptimal with respect to the mesh size $h$ by one order. \change{For instance, for $p=2$ and $k_T=3$, $T\in\mathcal{T}$, we have first order convergence with respect to the mesh-size $h$.} This is because the errors are measured in weaker norms tha\change{n} the ones appearing on the right-hand side of \eqref{eq:apriori_bound_hypoco}. We view this as a reasonable price to pay for the long time robustness of \change{the} method when $t_f\to \infty$. Indeed, in many practical scenarios, such as long time computations of decay to equilibrium distributions, the favourable exponential dependence with respect to $t_f$ may be preferable to a slight suboptimality in the rate of convergence with respect to the mesh size $h$. Moreover, given the typically high smoothness of the exact solution $u$, the use of high order finite element spaces may diminish further the practical significance of this slight $h$-suboptimality.  

Nevertheless, to address the aforementioned slight $h$-suboptimality in Theorem \ref{thm:error_analysis_two}, one may be tempted to consider scaling of the stabilisation terms involving $A$ by appropriate powers of the mesh-size $h$. This will reduce the stiffness matrix scaling to one of a second order operator; this is a standard practice in classical Galerkin contexts, e.g., for streamline upwinded Petrov-Galerkin methods for convection-dominated problems. In the present context, however, such scaling would introduce new challenges, most important of which is that the spectral gap constant $c_0$ would be proportional to a negative power of $h$. Although such dependence may not turn out to be catastrophic in certain scenarios, e.g., when $u_t$ decays very fast (cf., \eqref{eq:apriori_bound_hypoco} in Theorem \ref{thm:error_analysis_two}), this is a challenging question that requires further research.

Another possibility to retrieve optimality in the rate of convergence for the method \eqref{eq:FEM} in weaker norms is the incorporation of Aubin-Nitsche type duality techniques. \change{Using the} $h$-optimal error control of the (hypo)elliptic projection error \change{$\tnorm{u-\pi u}$, along with regularity estimates, it is conceivable that the order of convergence with respect to $h$ and $p$ in Theorem \ref{thm:error_analysis_two}  can be improved, since the $H^1$-seminorm is weaker than $\tnorm{\cdot}$}. Although this is classical in parabolic problems, it appears to be a challenge in the present context of hypocoercive operators. This is due to the non-closedness of the hypocoercivity property with respect to the adjoint operation in the present variational context \eqref{eq:weak_hypo}. This constitutes an interesting direction for future research.

\section{Numerical experiments}\label{sec:num}
\change{We now present some basic numerical experiments aiming to verify the convergence rates predicted by the theory. For the (not studied theoretically) time discretization, we have implemented a discontinuous Galerkin timestepping method of degree $p-2$ with $p$  being the local polynomial degree of the spatial discretization; this choice formally balances the expected rates of convergence, while maintaining a practical computational cost. For the experiments below we have made the following parameter choices: $t_f=0.5$, $\Omega=(0,1)^2$, $\kappa=\lambda=0$, $\alpha=0.35$, $\beta=\alpha^2$ , $\gamma=\alpha^3$, and $C_\tau=10$. We note that altering $\alpha$  within the permitted range does not appear to have essential impact on the errors observed.

\subsection{Example 1} We first consider the problem \eqref{eq:kolmogorov-bounded} and $u_0,\ f$ so that the exact solution is given by the smooth time-independent periodic function $u(t,x,y):=\sin^2(\pi x)\sin^2(\pi y)$. The purpose of this example is to observe the convergence of the spatial discretization.

\begin{figure}[t]
	\includegraphics[scale=0.287]{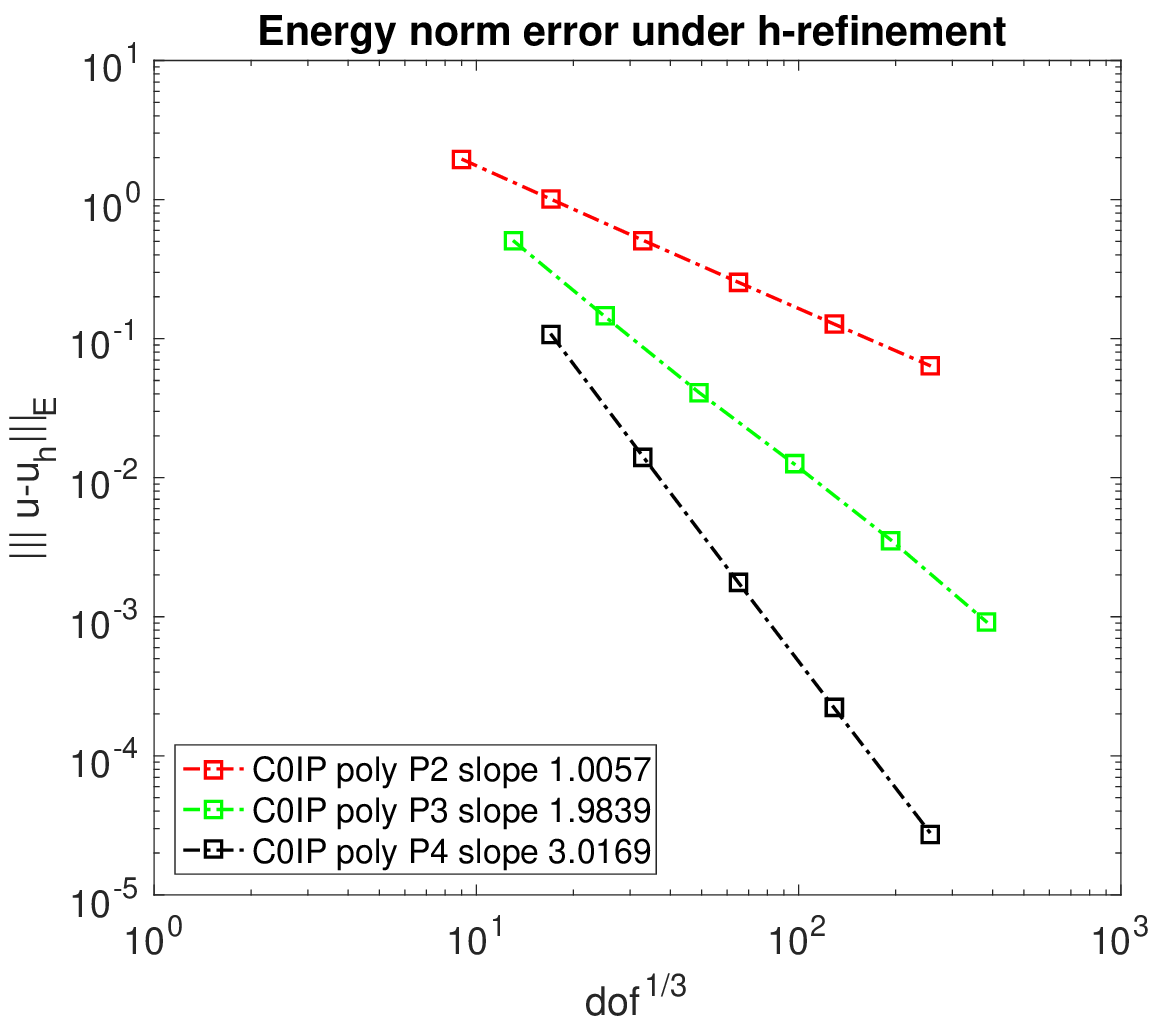}	\includegraphics[scale=0.287]{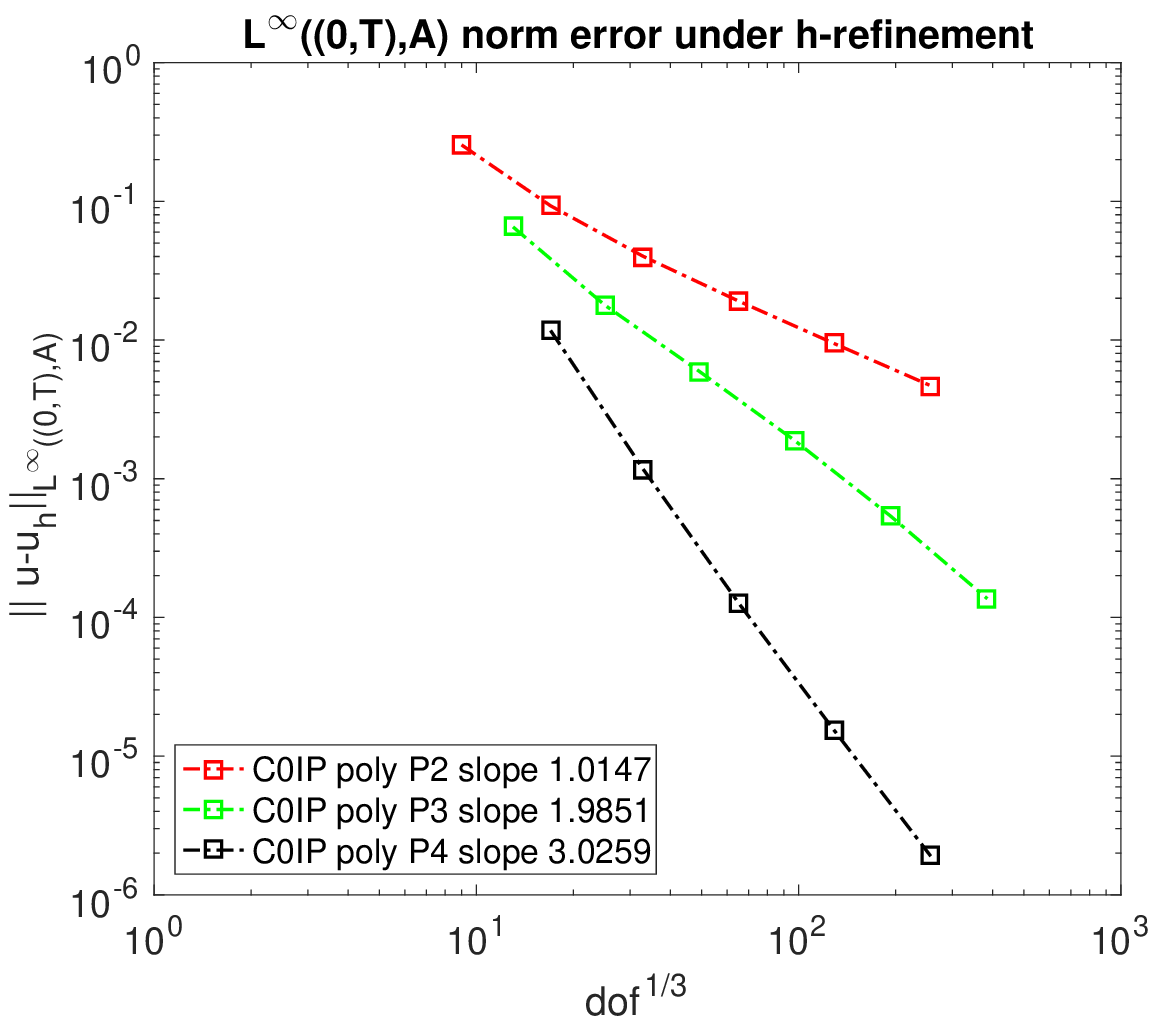}
	\includegraphics[scale=0.287]{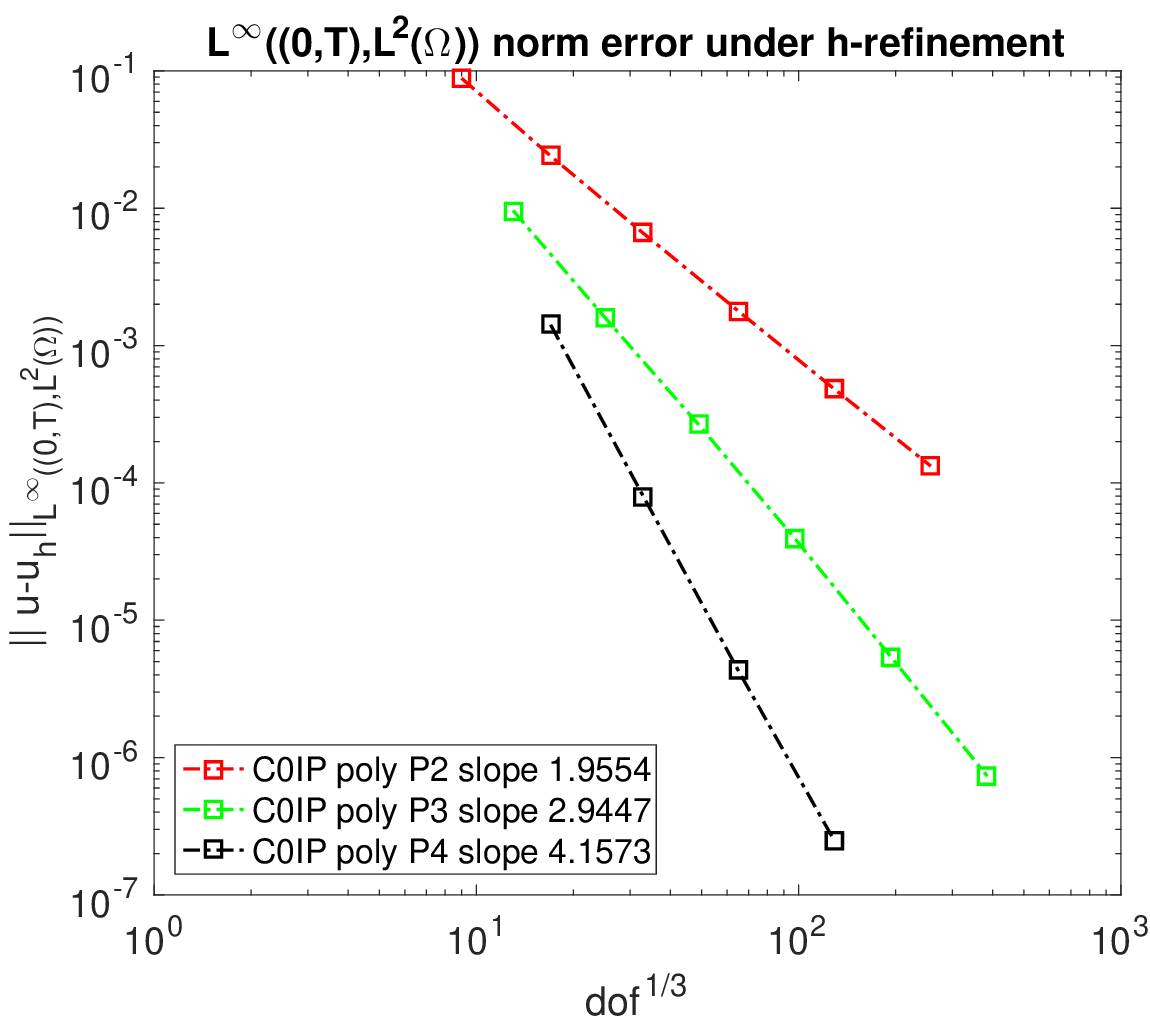}
	\caption{Example 1. Convergence history under $h$-refinement for the errors $\tnorm{e}{}$ (left), $\norm{\sqrt{A}\nabla e}{{L_\infty(0,t_f;L_2(\Omega))}}$ (centre), $\norm{e}{{L_\infty(0,t_f;L_2(\Omega))}}$ (right), for $p=2,3,4$, against the square root of total space-time degrees of freedom (Total DoF)$^{1/2}$.}
	\label{fig:example1}
\end{figure}
In Figure \ref{fig:example1}, the convergence rates under $h$-refinement are presented for the errors $\tnorm{e}{}$ (left), $\norm{\sqrt{A}\nabla e}{L_\infty(0,t_f;L_2(\Omega))}$ (centre), $\norm{e}{L_\infty(0,t_f;L_2(\Omega))}$ (right), and for local polynomial degrees $p=2,3,4$ against the square root of total space-time degrees of freedom (Total DoF)$^{1/2}$.  The space-time subdivision is constructed by tensorizing a sequence of spatial triangular meshes of $32$, $128$, $512$,  $2,\!048$,  $8,\!192$, and $32,\!768$  elements with the whole time interval $[0,t_f]$; there is no error in time, so one time step is sufficient. For the errors $\tnorm{e}{}$ and $\norm{\sqrt{A}\nabla e}{}$, we observe convergence rates $O(h^{p-1})$ as expected by the theory, while the error $\norm{e}{}$ converges at the higher rate 
$O(h^{p})$ although the theory currently only predicts   $O(h^{p-1})$. 

\subsection{Example 2}
 We now consider the problem \eqref{eq:kolmogorov-bounded} and $u_0,\ f$ so that the exact solution is given by the smooth function $u(t,x,y):=(2-t)^{-1}\exp(- (x - 1/2)^2 - (y - 1/2)^2)\sin(\pi (x-t))\sin^2(\pi y)$. Note that \eqref{eq:additional_bc_gen} is satisfied. The purpose of this experiment is to observe the convergence in the practical setting of the discontinuous Galerkin timestepping of degree $p-2$ in conjunction with $p$-th polynomial degree spatial discretization \eqref{eq:FEM}.

\begin{figure}[t]
	\includegraphics[scale=0.287]{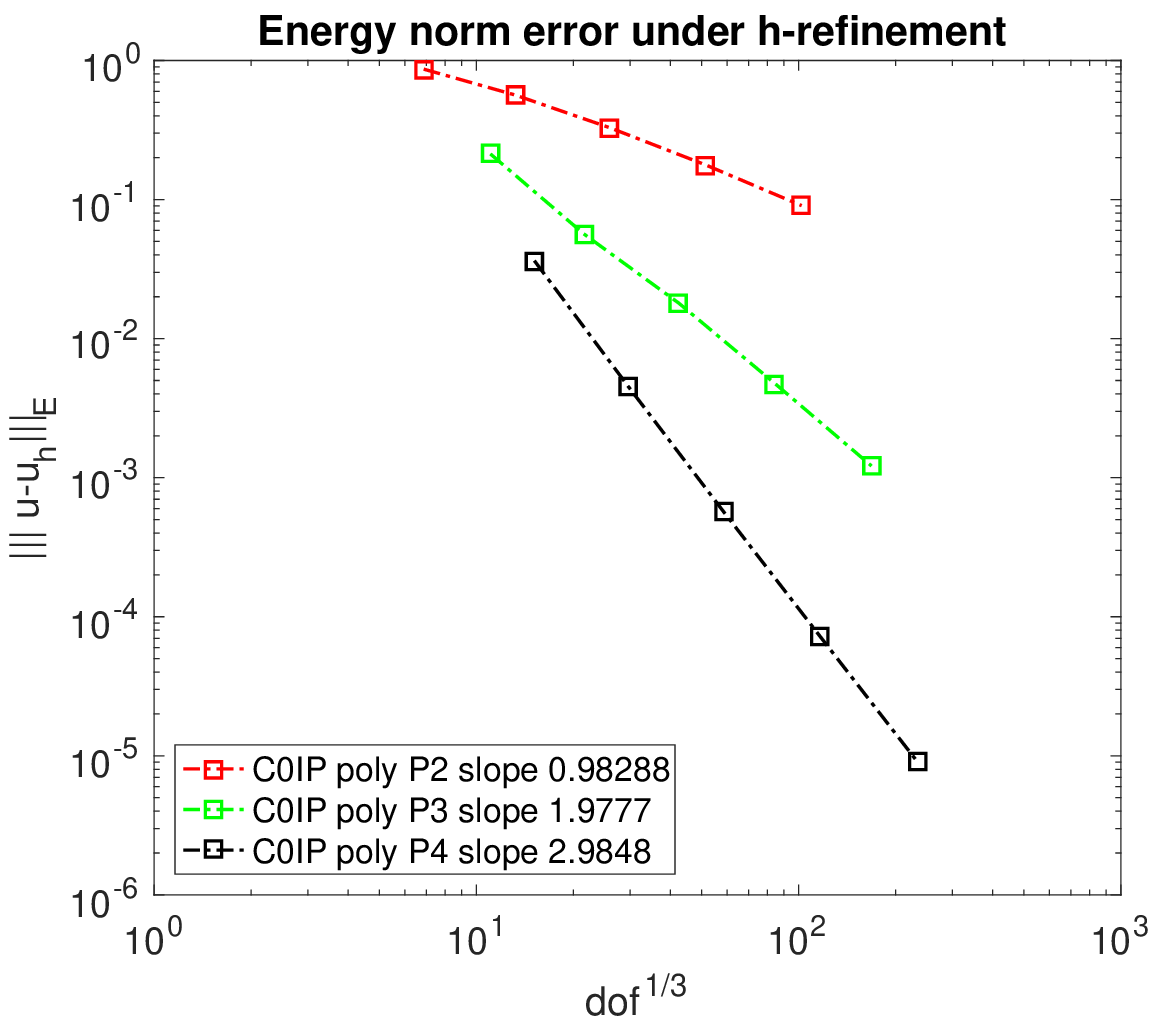}	\includegraphics[scale=0.287]{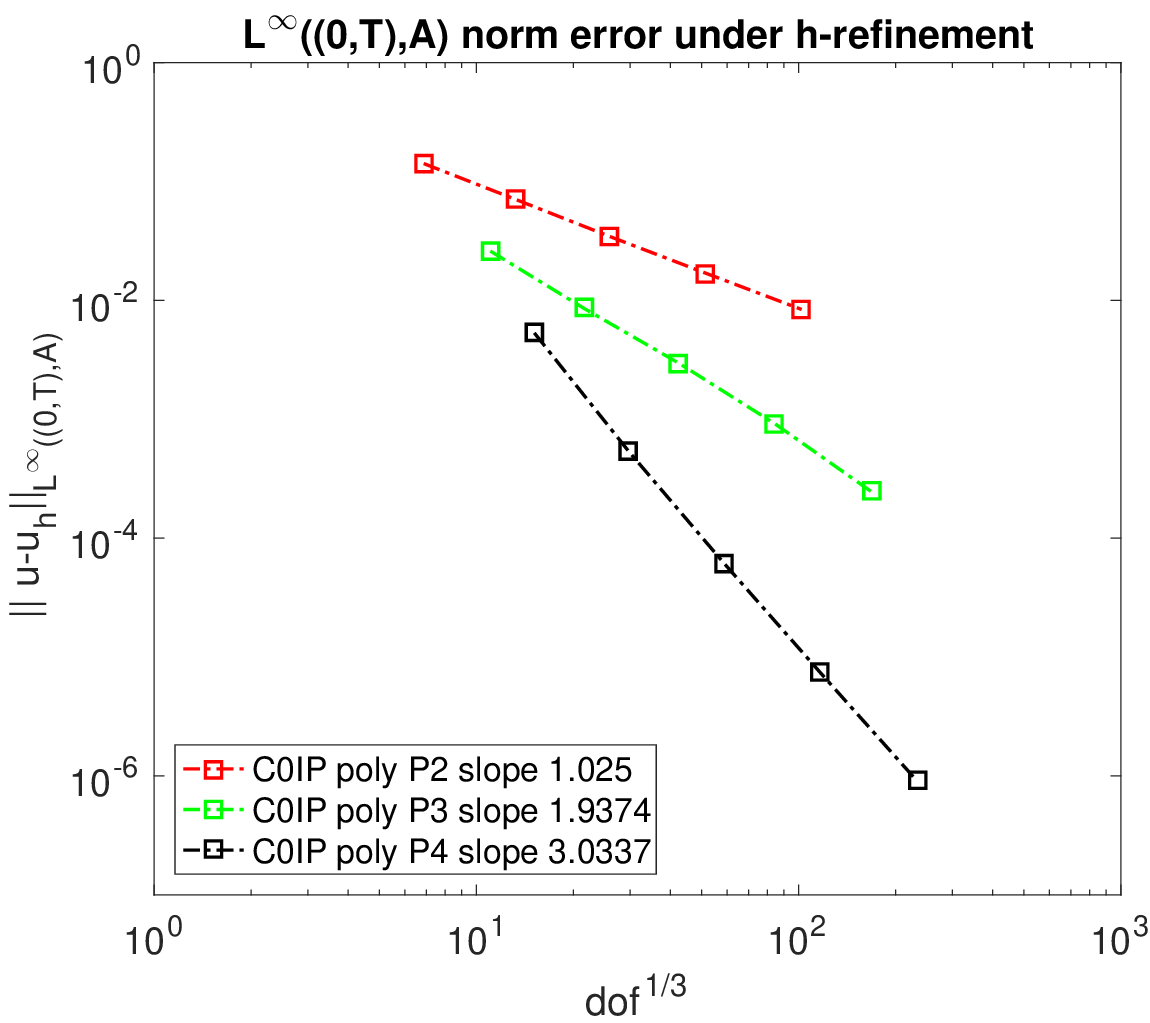}
	\includegraphics[scale=0.287]{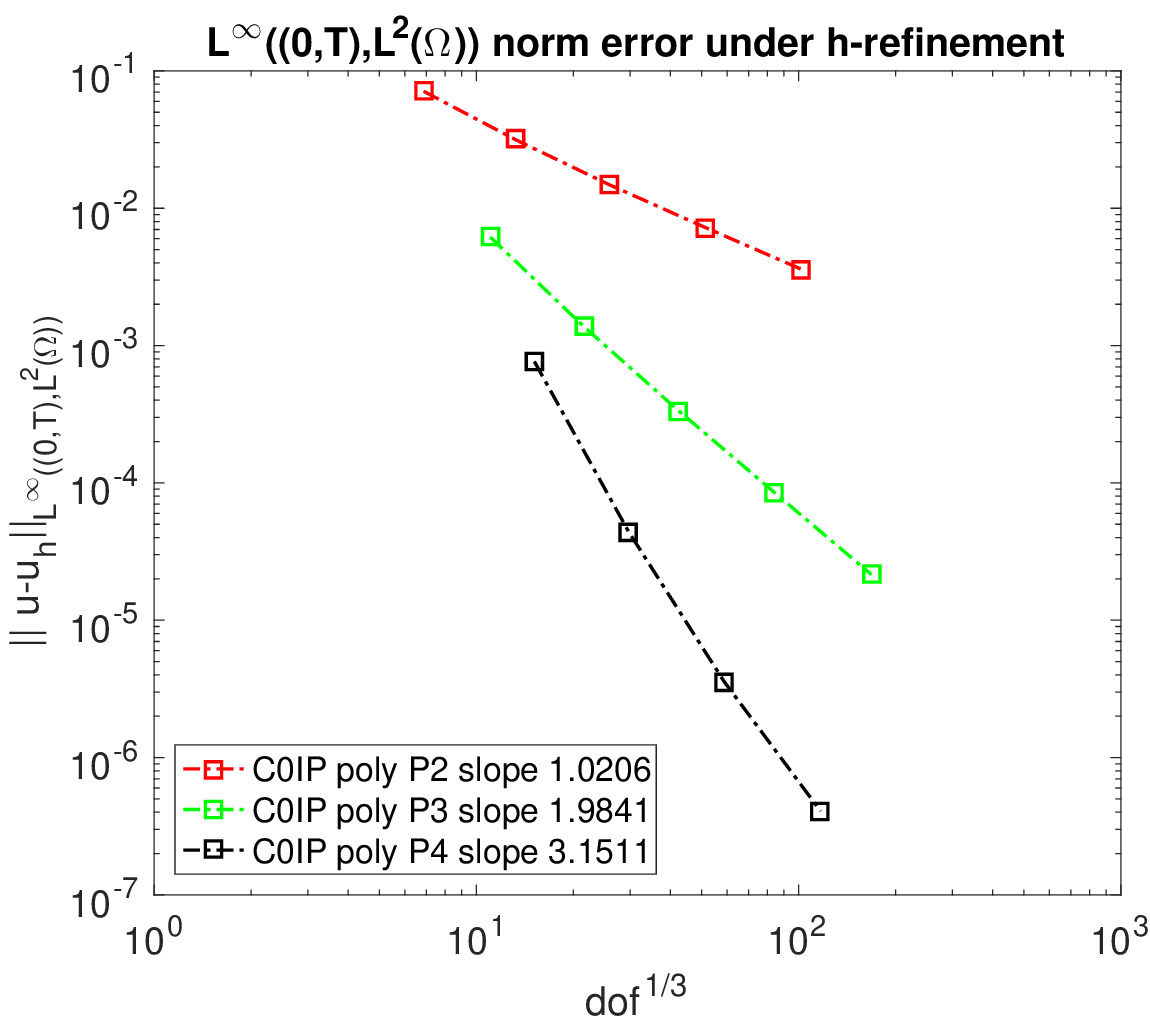}
	\caption{Example 2. Convergence history under $h$-refinement for the errors $\tnorm{e}{}$ (left), $\norm{\sqrt{A}\nabla e}{{L_\infty(0,t_f;L_2(\Omega))}}$ (centre), $\norm{e}{{L_\infty(0,t_f;L_2(\Omega))}}$ (right), for $p=2,3,4$, against the cubic root of total space-time degrees of freedom (Total DoF)$^{1/3}$.}
	\label{fig:example2}
\end{figure}

The space-time subdivision is constructed by tensorizing a sequence of spatial triangular meshes of $32$, $128$, $512$,  $2,\!048$,  and $8,\!192$  elements with respective subdivisions of $[0,t_f]$ into $4$, $8$, $16$, $32$, and $64$ time steps; notice that this choice corresponds to balanced spatial and temporal discretization parameters. In Figure \ref{fig:example2}, the convergence rates under $h$-refinement are presented for the errors $\tnorm{e}{}$ (left), $\norm{\sqrt{A}\nabla e}{L_\infty(0,t_f;L_2(\Omega))}$ (centre), $\norm{e}{L_\infty(0,t_f;L_2(\Omega))}$ (right), and for local polynomial degrees $p=2,3,4$ against the cubic root of total space-time degrees of freedom (Total DoF)$^{1/3}$.  We observe convergence rates $O(h^{p-1})$ for all errors. We expect that the reduced order of convergence for the error $\norm{e}{L_\infty(0,t_f;L_2(\Omega))}$ compared to Example 1 is due to the lower order temporal discretization.

We finally note that the behaviour of the proposed method in practical scenarios with regard to the dependence on the final time $t_f$ is a rather subtle and interesting topic in its own right and it will be discussed in detail elsewhere \cite{dg}.
}

\section*{Acknowledgements}
I \change{am sincerely grateful to} Zhaonan (Peter) Dong (Cardiff, UK) for helpful comments on an earlier version of this work \change{and for providing a computer code used for the numerical experiments}. The financial support of The Leverhulme Trust (research project grant no. RPG-2015-306) is also gratefully acknowledged.\change{ Also, this research work was supported by the Hellenic Foundation for Research and Innovation (H.F.R.I.) under the ``First Call for H.F.R.I. Research Projects to support Faculty members and Researchers and the procurement of high-cost research equipment grant" (Project Number: 3270).}

\end{document}